\documentclass[12pt]{article}
\title{Generic differentiability and $P$-minimal groups}
\author{Will Johnson}

\usepackage{amsmath, amssymb, amsthm}    	% need for subequations
\usepackage{fullpage} 	% without this, will have wide math-article-paper margins
\usepackage{amscd}
\usepackage{hyperref}
\usepackage[all]{xy}
\usepackage[T1]{fontenc}
\usepackage{lmodern}
\usepackage{centernot}
\usepackage{enumitem}
\usepackage{titlefoot}

\DeclareMathOperator*{\ind}{\raise0.2ex\hbox{\ooalign{\hidewidth$\vert$\hidewidth\cr\raise-0.9ex\hbox{$\smile$}}}}

\newcommand{\ter}{\operatorname{int}}

\newcommand{\ACVF}{\mathrm{ACVF}}

\newcommand{\ba}{\bar{a}}
\newcommand{\bb}{\bar{b}}
\newcommand{\bc}{\bar{c}}
\newcommand{\bx}{\bar{x}}
\newcommand{\by}{\bar{y}}
\newcommand{\bz}{\bar{z}}

\newcommand{\Th}{\operatorname{Th}}
\newcommand{\rad}{\operatorname{rad}}

\newcommand{\acl}{\operatorname{acl}}
\newcommand{\dcl}{\operatorname{dcl}}
\newcommand{\tp}{\operatorname{tp}}

\newcommand{\st}{\operatorname{st}}
\newcommand{\dom}{\operatorname{dom}}

\newcommand{\img}{\operatorname{im}}

\newcommand{\dpr}{\operatorname{dp-rk}}

\newtheorem{theorem}{Theorem}[section] % numbered like the section

\newtheorem{lemma}[theorem]{Lemma}

\newtheorem{corollary}[theorem]{Corollary}
\newtheorem{fact}[theorem]{Fact}

\newtheorem{conjecture}[theorem]{Conjecture}

\newtheorem{proposition}[theorem]{Proposition}
\newtheorem{proposition-eh}[theorem]{Proposition(?)}
\newtheorem*{theorem-star}{Theorem}
\newtheorem*{conjecture-star}{Conjecture}
\newtheorem*{lemma-star}{Lemma}
\newtheorem{claim}[theorem]{Claim}

\theoremstyle{definition}
\newtheorem{definition}[theorem]{Definition}
\newtheorem{example}[theorem]{Example}

\newtheorem{remark}[theorem]{Remark}
\newtheorem{observation}[theorem]{Observation}

\theoremstyle{remark}

\newtheorem*{acknowledgment}{Acknowledgments}

\newcommand{\Qq}{\mathbb{Q}}

\newcommand{\Rr}{\mathbb{R}}

\newcommand{\Zz}{\mathbb{Z}}

\newcommand{\Nn}{\mathbb{N}}

\newcommand{\Mm}{\mathbb{M}}
\newcommand{\Ff}{\mathbb{F}}
\newcommand{\Pp}{\mathbb{P}}
\newcommand{\Ll}{\mathcal{L}}
\newcommand{\Oo}{\mathcal{O}}
\newcommand{\mm}{\mathfrak{m}}

\newenvironment{claimproof}[1][\proofname]
               {
                 \proof[#1]
                 
               }
               {
                 \endproof
               }

\begin{document}

\maketitle\unmarkedfntext{

  }

\begin{abstract}
  We prove generic differentiability in $P$-minimal theories,
  strengthening an earlier result of Kuijpers and Leenknegt.  Using
  this, we prove Onshuus and Pillay's $P$-minimal analogue of Pillay's
  conjectures on o-minimal groups.  Specifically, let $G$ be an
  $n$-dimensional definable group in a highly saturated model $M$ of a
  $P$-minimal theory.  Then there is an open definable subgroup $H
  \subseteq G$ such that $H$ is compactly dominated by $H/H^{00}$, and
  $H/H^{00}$ is a $p$-adic Lie group of the expected dimension.
  Additionally, the generic differentiability theorem immediately
  implies a classification of interpretable fields in $P$-minimal
  theories, by work of Halevi, Hasson, and Peterzil.
\end{abstract}

\section{Introduction}
In the past few decades, o-minimality has been a major area of
research in model theory, with applications to machine learning
\cite{macson}, diophantine geometry \cite{andre-oort-1}, and Hodge
theory \cite{hodge-B}.  Much work has been done on definable groups in
o-minimal structures.  Pillay's \emph{o-minimal group conjectures}
\cite{conjecture} indicated a link between o-minimal definable groups
and real Lie groups.  Specifically, if $G$ is a \emph{definably
compact} definable group in an o-minimal structure, then $G$ should
have a smallest type-definable subgroup of small index $G^{00}$, and
the quotient group $G/G^{00}$ (which is naturally a compact Hausdorff
group) should be a real Lie group with the same dimension as $G$.
This conjecture was proven by Hrushovski, Peterzil, and Pillay
\cite{HPP}, building off a number of earlier papers.

O-minimality is closely tied to the theory RCF (real closed fields),
which is the complete theory of $(\Rr,+,\cdot)$.  Specifically, if $T$
is an expansion of RCF, then $T$ is \emph{o-minimal} if and only if
for every model $M \models T$, every 1-variable definable set $D
\subseteq M^1$ is already definable in the reduct $M \restriction
\mathcal{L}_{Rings} = (M,+,\cdot)$.  (Here and in what follows, ``definable'' always means ``definable with parameters.'')  In practice, most interesting
o-minimal theories are expansions of RCF.

Drawing inspiration from this point of view, Haskell and Macpherson
\cite{p-min} defined a ``version of o-minimality for the $p$-adics''
called \emph{$P$-minimality}.  Let $K$ be a $p$-adic field, i.e., a
finite extension of $\Qq_p$.  An expansion $T$ of $\Th(K)$ is
\emph{$P$-minimal} if for every model $M \models T$, every 1-variable
definable set $D \subseteq M^1$ is already definable in the reduct $M
\restriction \mathcal{L}_{Rings}$.  There are
non-trivial examples of $P$-minimal theories arising from analytic
functions on $\Qq_p$~\cite{analytic}.  Haskell, Macpherson, and later authors \cite{p-minimal-cells} show that
$P$-minimality has many properties similar to o-minimality, such as a
good dimension theory and topological cell decomposition theorem.\footnote{In addition to the ``topological cell decomposition'' of \cite{p-minimal-cells}, stronger cell decompositions are known under the assumption of definable Skolem functions \cite{mourgues,denef-cell,cluckers-cell,halupczok,cubides-delon}.  Thanks to the referee for providing this list of references.}

Onshuus and Pillay \cite{O-P} subsequently formulated a $P$-minimal
analogue of Pillay's o-minimal group conjectures.  Among other things,
their conjecture says that if $G$ is a $d$-dimensional definable group
in a $P$-minimal structure, then some $d$-dimensional definable
subgroup $H$ has the property that $H/H^{00}$ is a $d$-dimensional
$p$-adic Lie group (as opposed to a real Lie group).

In this paper, we prove Onshuus and Pillay's conjecture on definable
groups in $P$-minimal structures.  Moreover, we show that a certain key
property of o-minimal structures---the \emph{generic
differentiability} of definable functions---also holds in $P$-minimal
structures.

\subsection{Statement of results}

The first main result of this paper is the following generic
differentiability theorem:
\begin{theorem} \label{gd-thm}
  Let $M$ be a model of a $P$-minimal theory $T$.  Let $U \subseteq
  M^n$ be a non-empty definable open set.  Let $f : U \to M^m$ be a
  definable function.  Then there is a definable open set $U_0
  \subseteq U$ such that $\dim(U \setminus U_0) < n$ and $f$ is
  differentiable on $U_0$.
\end{theorem}
In fact, we even get generic \emph{strict} differentiability.  See
Definition~\ref{difdef} and Theorem~\ref{generic-diff}.  We also prove
an inverse function theorem for strictly differentiable functions (Theorem~\ref{ift}).

The generic differentiability theorem generalizes work of Kuijpers and
Leenknegt \cite[Theorem~1.8]{leenknegt}, who proved generic
differentiability in the special case where $T$ is ``strictly
$P$-minimal'', meaning that $T$ is the theory of some $P$-minimal
expansion of the $p$-adic field $K$.  Most $P$-minimal theories
arising in practice are strictly $P$-minimal.  However, there are many
simple examples of $P$-minimal theories which fail to be strictly
$P$-minimal, so it is conceptually better to have generic differentiability without
any additional assumptions.

Theorem~\ref{gd-thm} is analogous to the generic differentiability theorems in other settings, such as o-minimal expansions of RCF \cite[Chapter~7]{lou-o-minimality}, the \emph{Hensel minimal fields} of Cluckers, Halupczok, and Rideau-Kikuchi \cite[Lemma~5.3.5, Corollary~3.2.7]{hens-min1}, and definably complete \emph{C-minimal} fields of characteristic 0 \cite[Theorem~9.5]{own-C-min} (including the \emph{V-minimal} fields of Hrushovski and Kazhdan \cite[Corollary~5.17]{HK}).  The Hensel minimal case overlaps significantly with the $P$-minimal case, since $p$-adically closed fields and their expansions by analytic functions are both Hensel minimal and $P$-minimal.  On the other hand, there should be examples of $P$-minimal theories which are not Hensel minimal, using the argument of \cite[Example~1.1]{own-C-min}.

Halevi, Hasson, and Peterzil \cite{hhp} classified interpretable fields in
$P$-minimal structures with generic differentiability.  By
Theorem~\ref{gd-thm}, we can drop the assumption of generic
differentiability:
\begin{corollary}[{\cite[Theorem~1(3)]{hhp}}]
  If $M$ is a model of a $P$-minimal theory, then any infinite interpretable field in $M$
  is definably isomorphic to a finite extension of $M$.
\end{corollary}
As mentioned above, we will apply generic differentiability to prove the
Onshuus-Pillay conjecture on definable groups in $P$-minimal
structures.  Before stating the conjecture, we review some
machinery.  Recall that a theory $T$ is \emph{NIP} if no formula has the independence property (IP) \cite[Chapter~2]{NIPguide}.  Informally, $T$ is NIP if there are no big collections of independent definable sets, where a family of sets is ``independent'' if every boolean combination of sets in the family is non-empty.  O-minimal and $P$-minimal theories are NIP.  If $G$ is a definable group in a monster model $\Mm$,
recall that $G^{00}$ is the smallest type-definable subgroup of $G$
with small index, if it exists.  The group $G^{00}$ exists in NIP
theories
\cite[Theorem~8.4]{NIPguide}.  When $G^{00}$ exists, the quotient
group $G/G^{00}$ is naturally a compact Hausdorff group via the
\emph{logic topology}---the topology in which a set is closed iff it
pulls back to a type-definable subset of $G$.  As a compact Hausdorff group,
$G/G^{00}$ carries a normalized Haar measure.  A definable group $G$
in a monster model $\Mm$ is \emph{compactly dominated}
\cite[Definition~9.1]{HPP} if $G^{00}$ exists and the quotient map $f
: G \to G/G^{00}$ has the property that for any definable set $D
\subseteq G$, we can partition $G/G^{00}$ into three sets $D_0, D_1,
D_{\mathrm{err}}$ such that
\begin{gather*}
  f(x) \in D_0 \implies x \notin D \\
  f(x) \in D_1 \implies x \in D \\
  D_{\mathrm{err}} \text{ has Haar measure } 0.
\end{gather*}
Onshuus and Pillay conjectured the following in \cite{O-P}, by analogy
with Pillay's conjectures on o-minimal groups \cite{conjecture} (plus
the Compact Domination Conjecture of \cite{HPP}):
\begin{conjecture}[Onshuus-Pillay] \label{op-conj}
  Let $G$ be an $n$-dimensional definable group in a highly saturated
  $P$-minimal structure $M$.  Then there is an $n$-dimensional
  definable open subgroup $H \subseteq G$ with the following
  properties:
  \begin{enumerate}
  \item $H/H^{00}$ is isomorphic to an $n$-dimensional Lie group over
    $K$.
  \item $H$ is compactly dominated by $H/H^{00}$.
  \end{enumerate}
\end{conjecture}
Here $K$ is the finite extension of $\Qq_p$ whose theory $T$ expands.
We prove Conjecture~\ref{op-conj}, and in fact prove something a
little stronger:
\begin{theorem}[{Theorem~\ref{second-main}}] \label{intro-1}
  Let $G$ be an $n$-dimensional definable group in a highly saturated
  $P$-minimal structure $M$.  Then there is an $n$-dimensional
  definable open subgroup $H \subseteq G$ with the following
  properties:
  \begin{enumerate}
  \item $H$ is compactly dominated by $H/H^{00}$.
  \item $H/H^{00}$ is isomorphic to $(\Oo_K,+)^n$, where $\Oo_K$ is
    the ring of integers in $K$.  In particular, $H/H^{00}$ is
    isomorphic to an $n$-dimensional Lie group over $K$.
  \end{enumerate}
\end{theorem}
The proof strategy is a bit idiotic: one simply takes $H$ to be an
infinitesimally small neighborhood of 1 and checks that everything
works.  Generic differentiability helps control the structure of $H$,
showing that $(H,+)$ looks like a deformation of
$(\Oo^n,+)$.\footnote{Hrushovski, Peterzil, and Pillay use a very similar strategy in the proof of
\cite[Proposition~7.8]{HPP}, which gives a version of the o-minimal
group conjectures for ind-definable groups.  Like here, their strategy is to take an infinitesimal neighborhood of the identity.}

\subsection{Variants of Theorem~\ref{intro-1}}
While Theorem~\ref{intro-1} \emph{technically} resolves the
Onshuus-Pillay conjecture, it certainly goes against the spirit of the
conjecture---the intention was that $H$ would be close to $G$, and the
$p$-adic Lie group $H/H^{00}$ should have interesting structure
related to the structure of $G$ and $H$.  But the $H$ in
Theorem~\ref{intro-1} is much smaller than $G$, and the group
$H/H^{00}$ has a boring structure solely determined by $\dim(G)$.\footnote{Hrushovski, Peterzil, and Pillay \cite{HPP} make similar complaints against their Proposition~7.8, writing
\begin{quote}
  ``\ldots note that the locally compact quotient we obtained is abelian; it is indeed a locally compact manifestation of the Lie algebra of $G$.  We feel that the canonical compact quotient of a definably compact group $K$ reflects better the structure of $K$; for instance $K/K^{00}$ is non-abelian if $K$ is non-abelian.  In the general case too, there should also be a locally compact quotient whose structure is close to that of $G$.  We do not at the moment have a precise statement of this, either in the compact or in the locally compact cases'' \cite[\S 7]{HPP}.
\end{quote}}

In fact, we are essentially exploiting a loophole in Onshuus and
Pillay's specific formulation of their conjecture.  To be more
precise, I do \emph{not} know how to prove the following variant
conjecture:
\begin{conjecture}[Modified Onshuus-Pillay conjecture] \label{mod-op}
  Let $G$ be a definable group in a highly saturated
  $P$-minimal structure $M$ and $n = \dim(G)$.  Let $M_0 \prec M$ be a small model
  defining $G$.  Then there is an $n$-dimensional $M_0$-definable open
  subgroup $H \subseteq G$ with the following properties:
  \begin{enumerate}
  \item $H/H^{00}$ is isomorphic to an $n$-dimensional Lie group over
    $K$.
  \item $H$ is compactly dominated by $H/H^{00}$.
  \end{enumerate}
\end{conjecture}
(This is the same as Conjecture~\ref{op-conj}, except we are now
requiring $H$ to be defined over the same parameters that define $G$.)
The best I could prove was the following two theorems, each of which
has a drawback:
\begin{theorem}[{Theorem~\ref{third-main}}]\label{intro-2}
  In the setting of Conjecture~\ref{mod-op}, there is an
  $M_0$-definable open subgroup $H \subseteq G$ such that $H$ is
  compactly dominated by $H/H^{00}$, and $H/H^{00}$ is isomorphic to
  an $(n \cdot [K : \Qq_p])$-dimensional Lie group over $\Qq_p$.
\end{theorem}
So the case where $K = \Qq_p$ is completely handled, but when $[K :
  \Qq_p] > 1$ we get a slightly defective result, only getting a Lie
group over $\Qq_p$ rather than a Lie group over $K$.  The proof of
Theorem~\ref{intro-2} depends on the following deep results of
Lazard~\cite{lazard}:
\begin{enumerate}
\item[(i)] If $G$ is a Lie group over $\Qq_p$, then the topological group
  structure on $G$ determines the Lie group structure.
\item[(ii)] If $G$ is an abstract topological group, then $G$ is a Lie
  group over $\Qq_p$ if and only if $G$ satisfies a certain abstract
  group-theoretic condition (see \cite[Theorem~8.1]{app} for details).
\end{enumerate}
If we replace $\Qq_p$ with a finite extension $K$, then (i) certainly
fails.  This makes (ii) seem more unlikely, at least to a non-expert
like me.  Part (ii) is the specific thing we would need to generalize
Theorem~\ref{intro-2} to resolve Conjecture~\ref{mod-op}.

Using very different techniques, we also resolve
Conjecture~\ref{mod-op} in the special case where $T$ is the pure
theory of $p$-adically closed fields $\Th(K)$:
\begin{theorem}[{Theorem~\ref{last-main}}]\label{intro-3}
  Conjecture~\ref{mod-op} holds if $T$ is $\Th(K)$, the theory of pure
  $p$-adically closed fields.
\end{theorem}
The proof depends heavily on the fact that definable functions on $K$
are piecewise analytic \cite[Theorem~1.1]{vdDS}, something which has
no analogue in a general $P$-minimal theory.  Consequently, I do not
expect the proof of Theorem~\ref{intro-3} to usefully generalize.

\subsection{The proof of generic differentiability}
In proving generic differentiability (Theorem~\ref{gd-thm}), the key
technical tool is the ``definable compactness'' of closed bounded sets
(Proposition~\ref{dc-thm}).  Definable compactness was essentially proved
by Cubides-Kovacsics and Delon \cite{cubides-delon}, but we give a
slightly improved proof in Section~\ref{dc-sec}.  The improved version
isn't strictly necessary (Remark~\ref{gam-fam}), but makes the
subsequent proofs more streamlined.

To a first approximation, definable compactness ensures that the limit
\begin{equation*}
  \lim_{x \to a} \frac{f(x) - f(a)}{x - a}
\end{equation*}
exists, though it might take the value $\infty$ and might only be a
``one-sided limit''---the limit might depend on which ``direction''
$x$ approaches $a$ from.  To proceed further, we analyze the
asymptotic behavior of definable functions $f(x)$ as $x$ approaches 0.
In Section~\ref{sec-a}, we show that any definable function behaves
asymptotically like
\begin{gather*}
  f(x) \propto C x^q \text{ as } x \to 0 \tag{$\ast$} \\
   \text{that is, } \lim_{x \to 0} \frac{f(x)}{Cx^q} = 1
\end{gather*}
for some constant $C$ and rational number $q$, though again ($\ast$)
may only hold ``on one side'' of $0$.  To show this, we first analyze
the valuation $v(f(x))$, relying on Cluckers' theorem that the value
group is a pure model of Presburger arithmetic
\cite[Theorem~6]{cluckers}.  After controlling the valuation, we use
definable compactness to get the more precise statement ($\ast$).

To prove generic differentiability, one reduces to the case where
\begin{equation*}
  f(a + x) - f(a) \propto C(a) x^q \text{ as } x \to 0
\end{equation*}
for any $a$ in the domain of $f$, for some fixed $q$ independent of
$a$.  By carefully analyzing the behavior of $f$, one can show that
$q$ must be 1,\footnote{For example, you cannot have a definable
function $f$ such that $f(a + x) - f(a) \propto x^2$ at every point
$a$.  Contrast this with $\ACVF_{2,2}$, where the function $f(x) =
x^2$ \emph{does} have the asymptotic expansion $f(a + x) = f(a) +
x^2$.} and the asymptotic behavior of $f(a + x) - f(a)$ is
independent of which direction $x$ approaches $0$ from.  This gives
differentiability.

%% By extending these arguments, one can show that if $a$ is sufficiently
%% generic, then $f(x)$ has the expected asymptotic expansion around $a$:
%% \begin{equation*}
%%   f(a + x) \propto f(a) + f'(a)x + \frac{f'(a)}{2}x^2 + \frac{f''(a)}{6}x^3 + \cdots
%% \end{equation*}
%% See Appendix~\ref{taylor} for details.

\subsection{Conventions}
``Definable'' always means ``definable with parameters''.  The set of natural numbers $\Nn$ contains $0$.

We write group operations as $x \star y$, since we will frequently
need to distinguish from both addition $x+y$ and multiplication $x
\cdot y$.  We write inverses as $x^{-1}$ however.  (Usually it is
clear from context whether we mean the group inverse or the
multiplicative inverse.)

We generally work in the following setting: $K$ is $\Qq_p$ or a finite
extension, $\Th(K)$ is the theory of $K$ in the language of rings
$\mathcal{L}_{Rings}$, $T$ is a $P$-minimal expansion of $\Th(K)$, and
$\Mm$ is a highly saturated monster model of $T$.  If $M$ is an
elementary substructure of $\Mm$, or more generally a subfield of $\Mm$, then $\Oo_M$, $\mm_M$, $k_M$, and
$\Gamma_M$ denote the valuation ring, maximal ideal, residue field,
and value group, respectively.  When no subscript is given, assume $M
= \Mm$.  The valuation is written as $v(x)$ with the additive
conventions, so that
\begin{gather*}
  v(xy) = v(x) + v(y) \\
  v(x+y) \ge \min(v(x),v(y)).
\end{gather*}
The ball of radius $\gamma$ around $a$, written $B_\gamma(a)$ is the
\emph{closed} ball $\{x \in \Mm : v(x - a) \ge \gamma\}$.  The \emph{radius} of $B_\gamma(a)$ is $\gamma$.  We write the radius
of a ball $B$ as $\rad(B)$.  The \emph{parent} of $B_\gamma(a)$ is
$B_{\gamma-1}(a)$.  We say that $B$ is a \emph{child} of $B'$ if $B'$
is the parent of $B$.  Two balls are \emph{siblings} if they have the
same parent.  We say that $B$ is an \emph{ancestor} of $B'$ and $B'$
is a \emph{descendant} of $B$ if $B \supseteq B'$.

If $X$ is a set in a topological space, then $\overline{X}$ denotes the closure, $\ter(X)$ denotes the interior, and $\partial X$ denotes the frontier $\overline{X} \setminus X$.  If $\bx = (x_1,\ldots,x_n)$ is a vector, then $v(\bx)$ denotes
$\min(v(x_1),v(x_2),\ldots,v(x_n))$.

For $n$ a positive integer, we let $P_n$ denote the set of non-zero $n$th powers, and
$\overline{P_n}$ denote its closure, which is $P_n \cup \{0\}$.  A
famous theorem of Macintyre~\cite{macintyre} shows that $p$-adically
closed fields have quantifier elimination in the language of valued
fields expanded by the $P_n$ as unary predicates.\footnote{More precisely, Macintyre deals with the case of $\Th(\Qq_p)$.  For $\Th(K)$ when $K$ is a finite extension, the result is due to Prestel and Roquette \cite{Prestel-roquette}, and one must add new constant symbols to the language, naming certain parameters.  Thanks to the referee for clarifying this point.}

The definitions of differentiability and strict
differentiability will be reviewed in Section~\ref{review}.

\section{Definable compactness} \label{dc-sec}
If $M$ is any structure and $X$ is a definable topological space in
$M$, say that $X$ is \emph{definably compact} if for any downward-directed definable family $\mathcal{F}$ of closed non-empty subsets of
$X$, the intersection $\bigcap \mathcal{F}$ is non-empty.  More
generally, a definable subset $D \subseteq X$ is definably compact if
$D$ is definably compact with respect to the induced subspace
topology.  This notion was investigated independently by
Fornasiero~\cite{fornasiero} and the author
\cite[\S3.1]{wj-o-minimal}.

Definable compactness has the following good properties, proven in
\cite[\S3.1]{wj-o-minimal}; only (\ref{non-triv}) is
non-trivial.
\begin{fact}[{\cite[\S3.1]{wj-o-minimal}}] \label{dc-facts}
  ~
  \begin{enumerate}
\item If $X$ is Hausdorff and $D \subseteq X$ is definably compact,
  then $D$ is closed.
\item If $X$ is definably compact, any closed subset $X \subseteq D$
  is definably compact.
\item If $D$ is compact, then $D$ is definably compact.
\item \label{non-triv} A direct product of two definably compact sets
  is definably compact, with respect to the product topology.
\item A finite set is definably compact.
\item A finite union of definably compact sets is definably compact.
\item The image of a definably compact set under a continuous function
  is definably compact.
\item Definable compactness is preserved in elementary extensions: if
  $N \succeq M$, then $D(M)$ is definably compact iff $D(N)$ is
  definably compact.
\end{enumerate}
\end{fact}
Return to the setting of a $P$-minimal monster model $\Mm$.
\begin{lemma} \label{dc}
  The valuation ring $\Oo$ is definably compact.  In other words, any
  downward-directed definable family $\mathcal{F}$ of closed non-empty
  subsets of $\Oo$ has $\bigcap \mathcal{F} \ne \varnothing$.
\end{lemma}
This is \emph{almost} the same thing as Theorem (A) of Cubides-Kovacsics and Delon \cite{cubides-delon}, but slightly more general as we do
not require $\mathcal{F}$ to be nested (or indexed by $\Gamma$).
Luckily, it is not hard to deduce Lemma~\ref{dc} from their work.
\begin{proof}
  Suppose for the sake of contradiction that $\bigcap \mathcal{F} =
  \varnothing$.  Say that a ball $B \subseteq \Oo$ is \emph{good} if
  $B$ intersects every $C \in \mathcal{F}$.  Otherwise, say that $B$
  is \emph{bad}.  The following properties are straightforward:
  \begin{enumerate}
  \item \label{good1} The ball $\Oo$ is good.
  \item \label{good2} Every descendant of a bad ball is bad; every
    ancestor of a good ball is good.
  \item \label{good3} A ball is good if and only if one of its
    children is good.  This follows by directedness of $\mathcal{F}$,
    and the fact that each ball has only finitely many children.  In
    particular, every good ball has a good child.
  \item \label{good4} If $a \in \Oo$, then some ball $B_\gamma(a)$ is
    bad.  Otherwise, every ball around $a$ intersects every $C \in
    \mathcal{F}$.  As the sets in $\mathcal{F}$ are closed, this means
    $a \in \bigcap \mathcal{F}$, a contradiction.
  \end{enumerate}
  Let $\Theta \subseteq \Gamma$ be the set $\{\rad(B) : B \subseteq
  \Oo, ~ B \text{ is good}\}$.  By (\ref{good1}), $\Theta$ is
  non-empty.  By (\ref{good2}), $\Theta$ is downward-closed (in
  $\Gamma_{\ge 0}$).  By (\ref{good3}), $\Theta$ has no greatest
  element.  The value group $\Gamma$ is a pure model of Presburger
  arithmetic \cite[Theorem~6]{cluckers}, so it is definably complete.
  By definable completeness of $\Gamma$, $\Theta$ must be cofinal in
  $\Gamma$, and so $\Theta = \Gamma_{\ge 0}$.  In particular, there
  are arbitrarily small good balls (good balls of arbitrarily high
  radius).

  Let $f : \Oo \to \Gamma$ be the function
  \begin{equation*}
    f(a) = \max \{\gamma \in \Gamma : B_\gamma(a) \text{ is good}\}.
  \end{equation*}
  The set inside the maximum is non-empty by (\ref{good1}) and bounded
  above by (\ref{good4}), so the maximum exists by definable
  completeness of $\Gamma$.  Morally, the function $f(x)$ is
  measuring the distance from $x$ to $\bigcap \mathcal{F}$.

  The function $f$ is locally constant.  In fact, if $a \in \Oo$ and
  $\gamma_0 = f(a) + 1$, then $f$ is constant on the ball
  $B_{\gamma_0}(a)$.  Indeed, if $b \in B_{\gamma_0}(a)$, then
  \begin{gather*}
    B_{\gamma_0}(b) = B_{\gamma_0}(a) \text{ is bad} \\
    B_{\gamma_0 - 1}(b) = B_{\gamma_0 - 1}(a) \text{ is good},
  \end{gather*}
  and so $f(b) = \gamma_0 - 1 = f(a)$.

  By Fact~\ref{ckd-B} below, $f$ has a maximum value $\gamma_0$ on the
  set $\Oo$.  Then every ball of radius $> \gamma_0$ is bad,
  contradicting the fact that $\Theta$ is cofinal in $\Gamma$.
\end{proof}
\begin{fact}
  \label{ckd-B}
  Let $f : \Oo \to \Gamma$ be a definable function which is
  continuous, i.e., locally constant.  Then $f$ has a maximum.
\end{fact}
Fact~\ref{ckd-B} is an instance of Theorem (B) in
\cite{cubides-delon}.  The proof in \cite{cubides-delon} seems
somewhat complicated, so we give an alternate proof, which may be of
independent interest.  First we make a couple remarks.
\begin{remark}\label{modulus}
  If $\Delta \subseteq \Gamma$ is definable, say that an integer $n \in \Nn$ is a \emph{modulus} for $\Delta$ if $\Delta \cap (k +
  n\Gamma)$ is a finite union of intervals in $k + n\Gamma$ for $k = 0,
  1, 2, \ldots, n-1$.
  \begin{enumerate}
  \item If $\Delta \subseteq \Gamma$ is definable, then $\Delta$ is
    definable in $(\Gamma,+,\le)$, a model of Presburger
    arithmetic.\footnote{This follows by Cluckers' theorem that
    $\Gamma$ is a pure model of Presburger arithmetic
    \cite[Theorem~6]{cluckers}.  More trivially, it follows by pulling
    $\Delta$ back along the map $\Mm^\times \to \Gamma$, and using
    $P$-minimality.}  Quantifier elimination in Presburger arithmetic
    then implies that every definable $\Delta \subseteq \Gamma$ has a
    modulus.
  \item By compactness, the following is true: if $\{\Delta_a\}_{a \in
    X}$ is a definable family of subsets of $\Gamma$, then there is an
    integer $n$ which is uniformly a modulus for every set $\Delta_a$.
  \item If $\Delta$ is infinite and $n$ is a modulus for $\Delta$,
    then $\{\gamma,\gamma+n\} \subseteq \Delta$ for some
    $\gamma$. That is, $\Delta$ must contain two consecutive elements
    of $k + n\Gamma$ for some $k$.  Otherwise, $\Delta \cap (k +
    n\Gamma)$ is a finite union of points for every $k$, and $\Delta$
    is finite.
  \end{enumerate}
\end{remark}
\begin{remark} \label{ucon-0}
  Let $D$ be a \emph{finite} set and let $f : \Oo \to D$ be a definable
  function which is continuous (i.e., locally constant).  Then $f$ is
  uniformly continuous: there is $\gamma$ such that $f$ is constant on
  every ball of radius $\gamma$.  To see this, note that $f$ is
  definable in the reduct $M \restriction \mathcal{L}_{Rings}$ by
  $P$-minimality, so we can reduce to the base theory $\Th(K)$.  Then we
  can transfer the statement to the elementarily equivalent model $K$,
  where it holds by compactness of $\Oo_K$.
\end{remark}
\begin{proof}[Proof (of Fact~\ref{ckd-B})]
  Let $f : \Oo \to \Gamma$ be locally constant.  If $f$ is bounded
  above, then a maximum exists by definable completeness of $\Gamma$.
  Assume $f$ is unbounded, i.e., $f(\Oo)$ is cofinal in $\Gamma$.  Say
  a ball $B \subseteq \Oo$ is \emph{good$'$} if $f$ is unbounded on $B$,
  and bad$'$ otherwise.  Then $\Oo$ is good$'$, and every good$'$ ball has a
  good$'$ child.  As in the proof of Lemma~\ref{dc}, there are
  arbitrarily small good$'$ balls, i.e., good$'$ balls of radius $> \gamma$
  for any $\gamma$.

  Let $n$ be a uniform modulus for the sets
  \begin{equation*}
    f(B) = \{f(x) : x \in B\},
  \end{equation*}
  as $B$ ranges over descendants of $\Oo$.  Let $g(x)$ be $f(x)$ mod
  $2n$.  Then $g$ is a continuous function from $\Oo$ to the finite
  set $\Gamma/2n\Gamma$.  By Remark~\ref{ucon-0}, $g$ is uniformly
  continuous: there is $\gamma_0$ such that
  \begin{equation*}
    v(a - b) > \gamma_0 \implies g(a) = g(b) \iff f(a) \equiv f(b)
    \pmod{2n}.
  \end{equation*}
  Take a good$'$ ball $B$ of radius $> \gamma_0$.  Then $g$ is constant
  on $B$.  Because $B$ is good$'$, the set $f(B)$ is cofinal in $\Gamma$,
  hence infinite.  As $n$ is a modulus for $f(B)$, there are $\gamma,
  \gamma+n$ in $f(B)$ by Remark~\ref{modulus}(3).  That is, there are
  $a,b \in B$ such that
  \begin{equation*}
    f(a) + n = f(b).
  \end{equation*}
  Then $f(a) \not \equiv f(b) \pmod{2n}$, a contradiction.
\end{proof}

\begin{proposition} \label{dc-thm}
  Let $D$ be a definable subset of $\Mm^n$.  Then $D$ is definably
  compact if and only if $D$ is closed and bounded.
\end{proposition}
\begin{proof}
  If $D$ is not closed or not bounded, it is easy to see that $D$ is
  not definably compact.  Indeed, the proof from
  \cite[Lemma~2.4]{johnson-yao} works.

  Conversely, suppose $D$ is closed and bounded.  Then $D$ is a closed subset
  of an $n$-dimensional closed ball $B \subseteq \Mm^n$.  There is a definable homeomorphism between $B$ and $\Oo^n$.  By the properties of definable compactness discussed at the start of this section
  \begin{gather*}
    \Oo \text{ is definably compact} \implies \Oo^n \text{ is definably compact} \\
    \iff B \text{ is definably compact} \implies D \text{ is definably compact}. \qedhere
  \end{gather*}
\end{proof}

\subsection{Applications of definable compactness}

\begin{corollary} \label{ucon}
  Let $D \subseteq \Mm^n$ be closed and bounded.  Let $f : D \to
  \Mm^m$ be definable and continuous.  Then $f$ is uniformly
  continuous.
\end{corollary}
\begin{proof}
  Otherwise, there is some $\epsilon \in \Gamma$ such that for every
  $\delta \in \Gamma$, there are $\bar{a}, \bar{b} \in D$ with
  \begin{equation*}
    v(\bar{a}-\bar{b}) \ge \delta \text{ but } v(f(\bar{a}) -
    f(\bar{b})) \le \epsilon.
  \end{equation*}
  Here, $v(\bar{a})$ is short for $\min_{1 \le i \le n} v(a_i)$.

  Fix such an $\epsilon$.  For each $\delta$, let
  \begin{equation*}
    C_\delta = \{(\bar{a},\bar{b}) \in D^2 : v(\bar{a}-\bar{b}) \ge \delta \text{ and } v(f(\bar{a}) - f(\bar{b})) \le \epsilon\}.
  \end{equation*}
  Then $C_\delta$ is closed in $D^2$, by continuity of $f$.  By
  assumption, $C_\delta$ is non-empty.  The family
  $\{C_\delta\}_{\delta \in \Gamma}$ is nested, so by definable
  compactness of $D^2$, there is some $(\bar{a},\bar{b}) \in
  \bigcap_\delta C_\delta$.  Then
  \begin{equation*}
    v(\bar{a} - \bar{b}) = +\infty \text{ and } v(f(\bar{a}) -
    f(\bar{b})) \le \epsilon,
  \end{equation*}
  which is absurd.
\end{proof}

Let $f : D \to C$ be a definable function, for some definable sets $D
\subseteq \Mm^n$ and $C \subseteq \Mm^m$.  Let $x_0$ be a point in the
frontier $\partial D = \overline{D} \setminus D$.  Say that $b \in
\Mm^m$ is a \emph{cluster point} of $f(x)$ as $x \to x_0$ if for every
$\epsilon \in \Gamma$, there is $x \in B_\epsilon(x_0) \cap D$ such
that $f(x) \in B_\epsilon(b)$.  Equivalently, $(x_0,b)$ is in the
closure of the graph of $f$.
\begin{lemma} \label{cluster}
  In the above setting, suppose $C$ is definably compact.
  \begin{enumerate}
  \item There is at least one cluster point as $x \to x_0$.
  \item $\lim_{x \to x_0} f(x)$ exists if and only if there is a
    unique cluster point $b$, in which case $\lim_{x \to x_0} f(x) =
    b$.
  \end{enumerate}
\end{lemma}
\begin{proof}
  \begin{enumerate}
  \item For $\epsilon \in \Gamma$, let $G_\epsilon$ be the graph of
    $f$ restricted to $B_\epsilon(x_0) \cap D$.  By definable
    compactness\footnote{of $B_\epsilon(x_0) \times C$}, there is a
    point $(a,b) \in \bigcap_{\epsilon} \overline{G_\epsilon}$.  For
    any $\epsilon$ we have
    \begin{gather*}
      G_\epsilon \subseteq B_\epsilon(x_0) \times \Mm^m \\
      \overline{G_\epsilon} \subseteq B_\epsilon(x_0) \times \Mm^m \\
      a \in B_\epsilon(x_0).
    \end{gather*}
    Thus $a = x_0$.  Fix any $\epsilon$.  If $G$ is the graph of $f$,
    then
    \begin{equation*}
      (x_0,b) = (a,b) \in \overline{G_\epsilon} \subseteq \overline{G},
    \end{equation*}
    and so $b$ is a cluster point.
  \item First suppose $\lim_{x \to x_0} f(x)$ exists and equals $b$.
    Then \[(x_0,b) = \lim_{x \to x_0} (x,f(x)),\] and so $(x_0,b)$ is in
    the closure of the graph of $f$ and $b$ is a cluster point.
    Suppose for the sake of contradiction that $b'$ is another cluster
    point.  Take $\epsilon$ so large that $B_\epsilon(b') \cap
    B_\epsilon(b) = \varnothing$.  By existence of the limit, there is
    $\delta > \epsilon$ such that \[x \in B_\delta(x_0) \cap D
    \implies f(x) \in B_\epsilon(b).\] As $b'$ is a cluster point,
    there is $x \in B_\delta(x_0) \cap D$ such that $f(x) \in
    B_\delta(b') \subseteq B_\epsilon(b')$, contradicting the fact
    that $f(x) \in B_\epsilon(b)$.  This shows that $b$ is the unique
    cluster point.

    Conversely, suppose that $b$ is the unique cluster point.  We
    claim that $\lim_{x \to x_0} f(x) = b$.  Otherwise, there is some
    $\epsilon$ such that for any $\delta$, there is $x \in
    B_\delta(x_0) \cap D$ with $f(x) \notin B_\epsilon(b)$.  Fix
    such an $\epsilon$.  For each $\delta$, let $G_\delta$ be the
    non-empty definable set
    \begin{equation*}
      G_\delta = \{(x,f(x)) : x \in B_\delta(x_0) \cap D, ~ f(x)
      \notin B_\epsilon(b)\}.
    \end{equation*}
    By definable compactness, there is some $(a,b') \in
    \bigcap_{\delta} \overline{G_\delta}$.  As in part (1), $a = x_0$
    and $b'$ is a cluster point of $f$.  The set $G_\delta$ is
    contained in the closed set $\Mm^n \times (\Mm^m \setminus
    B_\epsilon(b))$ for each $\delta$, so we must have $b' \notin
    B_\epsilon(b)$.  Then $b' \ne b$, contradicting the fact that $b$
    is the unique cluster point. \qedhere
  \end{enumerate}
\end{proof}

\begin{remark} \label{gam-fam}
  A \emph{$\Gamma$-family} is a definable family of non-empty sets of
  the form $\{D_\gamma\}_{\gamma \in \Gamma}$, such that $\gamma <
  \gamma' \implies D_\gamma \subseteq D_{\gamma'}$.  Say that a
  definable topological space $X$ is \emph{$\Gamma$-compact} if every
  $\Gamma$-family of closed sets has non-empty intersection.  This
  condition is slightly weaker than our definition of definable
  compactness.

  In principle, we could use $\Gamma$-compactness rather than
  definable compactness, making Lemma~\ref{dc} unnecessary.  Theorem
  (A) of \cite{cubides-delon} says that closed bounded definable
  subsets of $\Mm$ are $\Gamma$-compact.  The applications of
  definable compactness, such as Corollary~\ref{ucon}, only use
  $\Gamma$-compactness.  Lastly, there is an analogue of
  Fact~\ref{dc-facts} for $\Gamma$-compactness.  Unfortunately, this
  analogue is a bit clumsy---one needs to restrict to definable
  topological spaces $X$ that are ``definably first countable'' in the
  sense that every point $a \in X$ has a $\Gamma$-family neighborhood
  basis.\footnote{In the base theory $\Th(K)$ of $p$-adically closed
  fields, \emph{every} definable (or interpretable) topological space
  is definably first countable, and in fact $\Gamma$-compactness
  agrees with definable compactness
  \cite[Theorem~8.11]{andujar-johnson}.  I don't know whether this
  continues to hold in $P$-minimal expansions.}  This isn't a problem in
  our case, because the spaces $\Mm^n$ are definably first countable.
\end{remark}

\section{More tools}
\subsection{Chains of nowhere dense sets}
The following lemma can be thought of as a ``definable Baire Category
Theorem''.
\begin{lemma} \label{baire}
  Let $\{D_\gamma\}_{\gamma \in \Gamma}$ be a
  $\Gamma$-family of subsets of $\Mm^k$, in the sense of
  Remark~\ref{gam-fam}.  Suppose every $D_\gamma$ is nowhere dense.
  Then the union $\bigcup_{\gamma \in \Gamma} D_\gamma$ is also
  nowhere dense.
\end{lemma}
\begin{proof}
  By dimension theory, the following are equivalent for a definable
  set $D \subseteq \Mm^k$:
  \begin{enumerate}
  \item $\dim(D) = k$.
  \item $D$ has non-empty interior.
  \item $\dim(\overline{D}) = k$.
  \item $\overline{D}$ has non-empty interior.
  \end{enumerate}
  Thus $D$ is nowhere dense iff $\dim(D) < k$.  We can rephrase the
  lemma as follows:
  \begin{quote}
    If $\{D_\gamma\}$ is a $\Gamma$-family of subsets of $\Mm^k$ and
    $\bigcup_{\gamma \in \Gamma} D_\gamma$ has dimension $k$, then
    some $D_\gamma$ has dimension $k$.
  \end{quote}
  By Fact~\ref{what}
  below, the condition $\dim\left(\bigcup_{\gamma \in \Gamma}
  D_\gamma\right) = k$ means that there are infinite subsets
  $S_1,\ldots,S_k \subseteq \Mm$ such that
  \begin{equation*}
    \bigcup_{\gamma \in \Gamma} D_\gamma \supseteq \prod_{i = 1}^k S_i.
  \end{equation*}
  We may assume the $S_i$ are countable (or merely small).  Then by
  saturation, we can take $\gamma$ sufficiently large that
  \begin{equation*}
    D_\gamma \supseteq \prod_{i = 1}^k S_i.
  \end{equation*}
  Again, by Fact~\ref{what}, this makes $D_\gamma$ have dimension $k$.
\end{proof}
\begin{fact}[{\cite[Lemma~2.3]{simon-walsberg}}] \label{what}
  If
  $D \subseteq \Mm^n$ is definable, then the following are equivalent:
  \begin{enumerate}
  \item $\dim(D) = n$.
  \item There are countable infinite sets $S_1, \ldots, S_n \subseteq
    \Mm$ such that $S_1 \times S_2 \times \cdots \times S_n \subseteq
    D$.
  \end{enumerate}
\end{fact}
%% Here, $\dpr(D)$ is the \emph{dp-rank} of $D$ \cite[Section~4.2]{NIPguide}, which agrees $\dim(D)$ in $P$-minimal theories \cite[Proposition~2.4]{simon-walsberg}.  Fact~\ref{what} is implicit in work of Simon \cite{surprise}, though I am having
%% trouble finding a precise reference.  At any rate, it is explicit in
%% \cite[Theorem~3.25]{prdf1a}, and for $P$-minimal theories also see \cite[Lemma~2.3]{simon-walsberg}.

\subsection{Strengthening limits}
Our goal is to prove Proposition~\ref{strengthen} below, which lets us
strengthen limits along one variable to full limits.  This will be
used later to convert differentiability into strict differentiability.
\begin{lemma} \label{bounce-angle}
  Let $U \subseteq \Mm^k$ be open, definable, and non-empty.  Let $X
  \subseteq \Mm^k \times \Mm^\times$ be a definable set such that $U
  \times \{0\} \subseteq \overline{X}$.  Then there is an open
  definable non-empty subset $U_0 \subseteq U$ such that $0 \in
  \overline{X_a}$ for every $a \in U_0$, where $X_a = \{b \in
  \Mm^\times : (a,b) \in X\}$.
\end{lemma}
\begin{proof}
  For any $\epsilon \in \Gamma$, let $D_\epsilon$ be the set of $a \in
  U$ such that $B_\epsilon(0) \cap X_a = \varnothing$.  Then
  $D_\epsilon$ has empty interior.  Otherwise, if $U_0$ is the
  interior of $D_\epsilon$, we have
  \begin{align*}
    (U_0 \times B_\epsilon(0)) \cap X &= \varnothing \\
    (U_0 \times B_\epsilon(0)) \cap \overline{X} &= \varnothing \\
    (U_0 \times \{0\}) \cap \overline{X} &= \varnothing,
  \end{align*}
  contradicting the fact that $U_0 \times \{0\} \subseteq U \times
  \{0\} \subseteq \overline{X}$.  Thus $D_\epsilon$ has empty interior
  as claimed.  Equivalently, $D_\epsilon$ is nowhere dense.  By Lemma~\ref{baire}, the
  union $\bigcup_\epsilon D_\epsilon$ has empty interior.  The
  complement $U \setminus \bigcup_\epsilon D_\epsilon$ then has
  non-empty interior.  Take $U_0$ to be the interior of $U \setminus
  \bigcup_\epsilon D_\epsilon$.  If $a \in U_0$, then $a \notin
  D_\epsilon$ for any $\epsilon$, which means $B_\epsilon(0) \cap X_a
  \ne \varnothing$, and $0 \in \overline{X_a}$.
\end{proof}

\begin{proposition} \label{strengthen}
  Let $U$ be a non-empty, open, definable subset of $\Mm^k$.  Let $D
  \subseteq \Mm$ be a definable set with $0 \in \partial D$.  Let $f :
  U \times D \to \Mm$ be a definable function.  Suppose that
  \begin{equation*}
    \lim_{y \to 0} f(a,y) \text{ exists}
  \end{equation*}
  for every $a \in U$.  Then there is a non-empty, open, definable
  subset $U_0 \subseteq U$ such that
  \begin{equation*}
    \lim_{(x,y) \to (a,0)} f(x,y) \text{ exists and equals } \lim_{y
      \to 0} f(a,y)
  \end{equation*}
  for every $a \in U_0$.
\end{proposition}
\begin{proof}
  Let $g(a) = \lim_{y \to 0} f(a,y)$ for $a \in U$.
  
  Note that the projective line $\Pp^1(\Mm)$ is definably
  compact, as it is covered by the two definably compact sets $\Oo$ and $\{x^{-1} : x \in \Oo\}$.  Regard $f$ as a function to $\Pp^1$.  For each $a \in U$,
  let $C_a$ be the set of cluster points of $f$ around $(a,0)$, i.e.,
  \begin{equation*}
    C_a = \{b \in \Pp^1(\Mm) : (a,0,b) \in \overline{G(f)}\} = \{b :
    (a,0,b) \in \partial G(f)\},
  \end{equation*}
  where $G(f)$ is the graph of $f$, regarded as a subset of $\Mm^{k+1}
  \times \Pp^1(\Mm)$.  Then $C_a$ is non-empty for each $a \in U$, by
  Lemma~\ref{cluster}.  In fact, clearly $g(a) \in C_a$ for each $a$.

  The set $C_a$ must be finite for generic $a \in U$.  Otherwise,
  $\dim \partial G(f) \ge \dim(U) + 1 = k + 1$, contradicting the
  small boundaries property \cite[Theorem~3.5]{p-minimal-cells}:
  \begin{equation*}
    \dim \partial G(f) < \dim G(f) = \dim(U \times D) = \dim U + \dim D = k+1.
  \end{equation*}
  So there is a non-empty, open definable subset $U_0 \subseteq U$
  such that $C_a$ is finite for $a \in U_0$.  Shrinking $U_0$ further,
  we may assume that the cardinality of $C_a$ is constant on $U_0$.

  If $|C_a| = 1$ for every $a \in U_0$, then $C_a = \{g(a)\}$ for each
  $a$, and Lemma~\ref{cluster} gives $\lim_{(x,y) \to (a,0)} f(x,y) = g(a)$.

  Suppose $|C_a| = N > 1$ for every $a \in U_0$.  $P$-minimal theories
  have definable finite choice \cite[Corollary~2.5]{p-minimal-cells},
  so there is a definable function $h$ such that $h(a) \in C_a
  \setminus \{g(a)\}$ for every $a \in U_0$.  $P$-minimal theories have
  generic continuity \cite[Theorem~4.6]{p-minimal-cells}, so by
  shrinking $U_0$ further, we may assume that $g$ and $h$ are
  continuous on $U_0$.  Fix some $a_0 \in U_0$ and let $V, W$ be
  clopen neighborhoods in $\Pp^1(\Mm)$ separating $g(a_0)$ and
  $h(a_0)$.  By continuity, we may shrink $U_0$ further and assume
  that
  \begin{gather*}
    g(a) \in V \\
    h(a) \in W
  \end{gather*}
  for every $a \in U_0$.  Let $X = \{(a,b) \in U \times D : f(a,b) \in
  W\}$.  The fact that $h(a) \in C_a$ means that $(a,0,h(a))$ is in
  the closure of the graph of $f$.  Then for any $\epsilon$, there are
  $(x,y) \in U \times D$ such that $x \in B_\epsilon(a)$, $y \in
  B_\epsilon(0)$, and $f(x,y) \in W$, i.e., $(x,y) \in X$.  This shows
  that $(a,0) \in \overline{X}$.  By Lemma~\ref{bounce-angle}, we may
  shrink $U_0$ further and ensure that $0 \in \overline{X_a}$ for
  every $a \in U_0$.  But the fact that $\lim_{y \to 0} f(a,y) = g(a)
  \in V$ implies that there is some $\epsilon$ such that if $y \in
  B_\epsilon(0)$, then $f(a,y) \in V$ and so $f(a,y) \notin W$, $(a,y)
  \notin X$, and $y \notin X_a$.  Then $B_\epsilon(0)$ shows that $0
  \notin \overline{X_a}$, a contradiction.
\end{proof}

\subsection{Rational powers}
Recall that $P_n$ denotes the set of non-zero $n$th powers in $\Mm$,
and $\overline{P_n}$ denotes its closure $P_n \cup \{0\}$, the set of
$n$th powers.  The following fact is well-known:
\begin{fact} \label{torsion-free}
  There is some $\ell$ such that $P_\ell$ is torsion-free.
\end{fact}
For example, this can be proven as follows.  Work in the standard model $K$.  If $I$ is a sufficiently small non-zero ideal in $\Oo_K$, then the multiplicative group $(1+I)^\times$ is isomorphic to a subgroup of $(K,+)$ via the $p$-adic logarithm map, and therefore $(1+I)^\times$ is torsion-free.  Meanwhile, $(1+I)^\times$ is an open subgroup of the compact group $\Oo_K^\times$, so the index $\ell = [\Oo_K^\times : (1+I)^\times]$ is finite.  Then the group of $\ell$th powers $P_\ell$ is isomorphic to $\Zz \times (1 + I)^\times$, which is torsion-free.
%% Work in the standard model $K$.  Take $I$ so small that $(1+I)^\times$
%% is isomorphic to an additive group, via the $p$-adic
%% exponential/logarithm.  Then $(1+I)^\times$ is torsion-free.  Let $q$
%% be the size of the residue field.  The map $x \mapsto x^{q-1}$ sends
%% $\Oo^\times$ to $1 + \mm$.  The $p$th power map is contracting on $1 +
%% \mm$, and the valuation is discrete, so the map $x \mapsto x^{p^n}$
%% will send $1 + \mm$ to $1 + I$ if $n$ is big enough.  Then the map $x
%% \mapsto x^{(q-1)p^n}$ sends $\Oo^\times$ into $1 + I$.  Take $\ell =
%% (q-1)p^n$.  The set of $\ell$th powers in $\Oo^\times$ is contained in
%% $1+I$, so it is torsion-free.  As $K^\times \cong \Oo^\times \times
%% \Zz$, the set of $\ell$th powers in $K^\times$ is also torsion-free.
\begin{definition} \label{qp-def}
  Let $q = a/b$ be a rational number in lowest terms.  A \emph{$q$th
  power map} on $P_n$ is a multiplicative homomorphism
  \begin{equation*}
    f : P_n \to \Mm^\times
  \end{equation*}
  such that $f$ is 0-definable in the pure language of rings, and
  $f(x)^b = x^a$.
\end{definition}
When such an $f$ exists, we write it as $x^q$ rather than
$f(x)$.\footnote{Multiple $q$th power maps may exist on $P_n$, but the
ambiguity won't cause us any problems.  We will construct some
canonical $q$th power maps in Lemma~\ref{qth-power} below.  If you
like, you may assume that $x^q$ only means the map constructed in
Lemma~\ref{qth-power}.}
\begin{example}
  If $K$ contains $\sqrt{-1}$, then there is no square root map
  $x^{1/2}$ on $P_2$.  Otherwise,
  \begin{equation*}
    1 = 1^{1/2} = (-1)^{1/2}(-1)^{1/2} = -1.
  \end{equation*}
\end{example}
Note that if there is a $q$th power map on $P_n$, then there is a
$q$th power map on $P_m$ for any $n \mid m$, simply by restricting
from $P_n$ to $P_m$.  The next lemma says that for fixed $q$, there is
a $q$th power map on $P_n$ for all sufficiently divisible $n$.
\begin{lemma} \label{qth-power}
  For any rational number $q = a/b$, there is some $n$ such that there
  is a $q$th power map on $P_n$.
\end{lemma}
\begin{proof}
  Let $P_\ell$ be torsion-free, as in Fact~\ref{torsion-free}.  Then
  the $b$th power map $P_\ell \to P_{\ell b}$ is a 0-definable
  bijection.  Let $g : P_{\ell b} \to P_\ell$ be its inverse.  Let
  $x^q$ be the map
  \begin{align*}
    P_{\ell b} &\to P_{\ell a} \\
    x &\mapsto g(x)^a.  \qedhere
  \end{align*}
\end{proof}

\section{Asymptotic behavior of functions in one variable} \label{sec-a}
In this section, we look at the asymptotic behavior of definable
functions
\begin{equation*}
  f : P_n \to \Mm
\end{equation*}
as $x \in P_n$ approaches 0.
We show that after restricting to a
smaller set $P_m \subseteq P_n$, the function $f$ looks asymptotically
like $Cx^q$ for some constant $C \in \Mm$ and rational number $q$:
\begin{equation*}
  \lim_{x \in P_m, ~ x \to 0} \frac{f(x)}{Cx^q} = 1.
\end{equation*}
The following fact is well-known, and can be deduced from
\cite[Lemma~4]{other-clucker}, \cite[Theorem~7.3]{denef-84}, or \cite[Theorem
  1.8]{denef-cell}.\footnote{By
$P$-minimality, one may assume that $D$ is definable in the pure field
reduct, so the results of \cite{other-clucker,denef-84,denef-cell} on
semialgebraic sets apply.  The results in
\cite{other-clucker,denef-84,denef-cell} are statements about multivariable
definable sets in the standard model $K$ rather than 1-variable sets
in the monster model $\Mm \succ K$, but they imply Fact~\ref{y-u-no} in
the same way that the cell decomposition theorem in an o-minimal
structure $M$ implies the o-minimality of elementary extensions $N
\succ M$.  Theorem~7.3 in \cite{denef-84} implies the
statement that given any finite set of polynomials $f_1,\ldots,f_m \in
\Mm[t]$ and a positive integer $n$, we can partition $\Mm$ into finitely many
annuli $A$ with centers $c$ such that on $A$, each $f_i$ has the form
$u_i(t)^n \cdot h_i \cdot (t - c)^{\nu_i}$ with $u_i(t)$ a definable
function taking values in $\Oo^\times$ and some $h_i \in \Mm$.  Here, an ``annulus'' means a set as in Fact~\ref{y-u-no}($\ast$) with $m=1$.  For $k \mid n$, the truth
value of $P_k(f_i(t))$ on $A$ is determined by the coset $(t-c)P_n$
in $\Mm/P_n$.  So if we further split each annulus $A$ into the
cosets $A \cap \lambda P_n$, then the truth value of $P_k(f_i(t))$ is
constant on each coset.  Combined with Macintyre's quantifier
elimination, this gives Fact~\ref{y-u-no}.  This proof is implicit in the proof of \cite[Theorem~7.4]{denef-84}.}
\begin{fact}[1-dimensional cell decomposition]\label{y-u-no}
  If $D \subseteq \Mm$ is definable, then $D$ can be partitioned into
  finitely many sets of the form
  \begin{equation*}
    \{x \in \Mm : \gamma_1 \Box_1 v(x - c) \Box_2 \gamma_2, ~ x - c \in \lambda P_m\} \tag{$\ast$}
  \end{equation*}
  where $c \in \Mm$, $\gamma_1, \gamma_2 \in \Gamma$, $\lambda \in
  \Mm$, and each relation $\Box_1, \Box_2$ is either $<$ or
  no condition.
\end{fact}
Fact~\ref{y-u-no} easily implies the following:
\begin{corollary} \label{asymptot0}
  Let $S$ be a finite set and $f : B_\epsilon(0) \cap P_n \to S$ be a
  definable function.  Then there is $\delta \ge \epsilon$ and $m$ a
  multiple of $n$ such that the restriction
  \begin{equation*}
    f : B_\delta(0) \cap P_m \to S
  \end{equation*}
  is constant.
\end{corollary}
\begin{lemma} \label{asymptot1}
  Let $f : B_\epsilon(0) \cap P_n \to \Oo^\times$ be a definable
  function.  Then there is $\delta \ge \epsilon$ and $m$ a multiple of
  $n$ such that the restriction
  \begin{equation*}
    f : B_\delta(0) \cap P_m \to \Oo^\times
  \end{equation*}
  extends to a continuous function
  \begin{equation*}
    g : B_\delta(0) \cap \overline{P_m} \to \Oo^\times.
  \end{equation*}
  Equivalently, $\lim_{x \to 0} f(x)$ exists, if we restrict $x$ to
  $P_m$.
\end{lemma}
\begin{proof}
  Let $G$ be the graph of $f$.  Let $C$ be the set of cluster points
  of $f$ at $x = 0$:
  \begin{align*}
    C &= \{b : (0,b) \in \overline{G}\} \\
    &= \{b : (0,b) \in \partial G\}.
  \end{align*}
  The set $G$ has dimension 1, so its frontier has dimension 0 by the
  small boundaries property \cite[Theorem~3.5]{p-minimal-cells}.  Therefore $C$ is finite.  On the
  other hand, $C$ is non-empty by Lemma~\ref{cluster}, and definable
  compactness of $\Oo^\times$.  Let $C = \{b_1,\ldots,b_\ell\}$.  Take
  pairwise disjoint clopen neighborhoods $V_i \ni b_i$, and let $V_0$
  be the complement $\Oo^\times \setminus (V_1 \cup \cdots \cup V_n)$.
  Let $h(x)$ be the unique $i \in \{0,\ldots,n\}$ such that $f(x) \in
  V_i$.  By Corollary~\ref{asymptot0}, there are $\delta \ge \epsilon$ and
  $m \in n\Zz$ such that $h(x)$ is a constant value $i$ on
  $B_\delta(0) \cap P_m$.  Then
  \begin{equation*}
    f(x) \in V_i \text{ for } x \in B_\delta(0) \cap P_m.
  \end{equation*}
  Let $f'$ be the restriction of $f$ to $B_\delta(0) \cap P_m$ and let
  $C'$ be the set of cluster points of $f'$ at $0$.  Again, $C'$ is
  non-empty by Lemma~\ref{cluster}.  Clearly $C' \subseteq C$.
  Because the graph of $f'$ lies inside the closed set $\Mm \times
  V_i$, we must have $C' \subseteq V_i$.  Then
  \begin{equation*}
    C' \subseteq C \cap V_i = 
    \begin{cases}
      \varnothing & \text{ if } i = 0 \\
      \{b_i\} & \text{ if } 1 \le i \le n.
    \end{cases}
  \end{equation*}
  As $C'$ is non-empty, we must have $i > 0$ and $C' = \{b_i\}$.  Then
  Lemma~\ref{cluster} gives $\lim_{x \to 0} f'(x) = b_i$.
\end{proof}
By a theorem of Raf Cluckers \cite[Theorem~6]{cluckers}, the value
group $\Gamma$ is a pure model of Presburger arithmetic.  Therefore,
$\Gamma$ is stably embedded, and eliminates imaginaries.

Cluckers also proves a cell decomposition theorem for definable sets
in Presburger arithmetic.  For $i_1,\ldots,i_n \in \{0,1\}$, he
defines a class of ``$(i_1,\ldots,i_n)$-cells'' in $\Gamma^n$, and
shows that any definable set in $\Gamma^n$ can be partitioned into
finitely many cells.  We only need the cases $n = 1, 2$.  For subsets
of $\Gamma^1$, a (0)-cell is a singleton $\{a\}$, and a (1)-cell is an
infinite set of the form $[a,b] \cap (k + n\Gamma)$, where $a,b \in (k
+ n\Gamma) \cup \{\pm\infty\}$.  For subsets of $\Gamma^2$,
\begin{itemize}
\item A $(0,0)$-cell is a singleton $\{(a,b)\}$.
\item A $(0,1)$-cell is a vertical segment $\{a\} \times C$, where $C
  \subseteq \Gamma$ is a (1)-cell.
\item A $(1,0)$-cell is the graph of a linear function on a
  $(1)$-cell.
\item There is also a notion of $(1,1)$-cell.
\end{itemize}
We don't need the precise definition of $(1,1)$-cell; the only
important thing to know is that $(1,1)$-cells have $\acl$-rank 2, or
equivalently, dp-rank 2.\footnote{Recall that \emph{dp-rank} \cite[Section~4.2]{NIPguide} is a relatively robust notion of ``dimension'' or ``rank'' for definable sets in NIP theories.  Algebraic closure satisfies exchange in
Presburger arithmetic, so dp-rank agrees with $\acl$-rank \cite[Theorem~0.3(0)]{surprise}.  We will use dp-rank to transfer dimension information from $\Mm$ to other sorts.  The important things to know are that $\dpr(X \times Y) = \dpr(X) + \dpr(Y)$, $\dpr(X \cup Y) = \max(\dpr(X),\dpr(Y))$, and $\dpr(\img(f)) \le \dpr(\dom(f))$ for any definable function $f$.  On the home sort $\Mm$, dimension agrees with dp-rank \cite[Proposition~2.4]{simon-walsberg}.}
Therefore, if $D \subseteq \Gamma^2$ is a definable set with dp-rank
$\le 1$, then $D$ is a finite union of $(0,0)$-cells, $(0,1)$-cells,
and $(1,0)$-cells.
\begin{lemma} \label{asymptot2}
  Let $f : B_\epsilon(0) \cap P_n \to \Mm^\times$ be a definable
  function.  Then there is $\delta \ge \epsilon$, $m$ a multiple of
  $n$, a rational number $q$, and an element $\gamma_0 \in \Gamma$
  such that
  \begin{equation*}
    v(f(x)) = q \cdot v(x) + \gamma_0
  \end{equation*}
  for $x \in B_\delta(0) \cap P_m$.
\end{lemma}
\begin{proof}
  The graph of $f$ has dimension 1, or equivalently, dp-rank 1.  Let
  $\Theta$ be its image under the valuation map:
  \begin{equation*}
    \Theta = \{(v(x),v(f(x))) : x \in B_\epsilon(0) \cap P_n\}
    \subseteq \Gamma^2.
  \end{equation*}
  Then $\Theta$ has dp-rank at most 1, so it is a union
  \begin{equation*}
    \Theta = C_1 \cup \cdots \cup C_\ell,
  \end{equation*}
  where each $C_i$ is a $(0,0)$-cell, a $(0,1)$-cell, or a
  $(1,0)$-cell.  Let $h : B_\epsilon(0) \cap P_n \to
  \{1,\ldots,\ell\}$ be the function such that $h(x) = i$ iff
  $(v(x),v(f(x))) \in C_i$.  By Corollary~\ref{asymptot0}, there are
  $\delta \ge \epsilon$ and $m \in n\Zz$ such that $h$ is constant on
  $B_\delta(0) \cap P_m$.  Then there is a cell $C$ such that
  \begin{equation*}
    x \in B_\delta(0) \cap P_m \implies (v(x),v(f(x))) \in C.
  \end{equation*}
  The set
  \begin{equation*}
    \{v(x) : x \in B_\delta(0) \cap P_m\}
  \end{equation*}
  has no upper bound, so there can be no upper bound on the first
  coordinate of points in $C$.  This prevents $C$ from being a
  $(0,0)$-cell or a $(0,1)$-cell.  It must instead be a $(1,0)$-cell,
  i.e., the graph of a linear function $qx + \gamma_0$ on a $(1)$-cell
  $D \subseteq \Gamma$.  Then
  \begin{equation*}
    x \in B_\delta(0) \cap P_m \implies (v(x),v(f(x))) \in C \implies v(f(x)) = q \cdot v(x) +
    \gamma_0.
  \end{equation*}
  A priori, $\gamma_0$ could belong to the divisible hull $\Gamma
  \otimes_\Zz \Qq$ rather than $\Gamma$ itself.  Write $q$ as $a/b$
  with $a,b \in \Zz$.  Take $x \in B_\delta(0) \cap P_m \cap P_b$.
  Then $q \cdot v(x)$ and $v(f(x))$ are both in $\Gamma$, implying
  that $\gamma_0 \in \Gamma$.
\end{proof}

\begin{proposition} \label{asymptot3}
  Let $f : B_\epsilon(0) \cap P_n \to \Mm$ be a definable function.
  Then there is $\delta \ge \epsilon$ and $m$ a multiple of $n$ such
  that one of the following holds:
  \begin{enumerate}
  \item $f$ is identically zero on $B_\delta(0) \cap P_m$.
  \item After restricting $f$ to $B_\delta(0) \cap P_m$, there is a
    rational number $q$ and constant $C \in \Mm^\times$ such that
    $P_m$ has a $q$th power map and
    \begin{equation*}
      \lim_{x \to 0} \frac{f(x)}{Cx^q} = 1.
    \end{equation*}
  \end{enumerate}
\end{proposition}
\begin{proof}
  Let
  \begin{equation*}
    h(x) = 
    \begin{cases}
      1 & \text{ if } f(x) = 0 \\
      0 & \text{ if } f(x) \ne 0.
    \end{cases}
  \end{equation*}
  By increasing $\epsilon$ and $m$, we may assume that $h$ is constant
  (Corollary~\ref{asymptot0}).  Then $f$ is everywhere zero or everywhere
  non-zero.  In the first case, we are done.  Assume the second case.
  Lemma~\ref{asymptot2} gives $\gamma_0 \in \Gamma$ and $q \in \Qq$
  such that
  \begin{equation*}
    v(f(x)) = q \cdot v(x) + \gamma_0
  \end{equation*}
  after possibly restricting the domain of $f$.  Take $C_0$ with
  $v(C_0) = \gamma_0$.  Increasing $m$, we may assume $P_m$ has a
  $q$th power map (Lemma~\ref{qth-power}).  Then
  \begin{equation*}
    v\left(\frac{f(x)}{C_0 x^q}\right) = 0
  \end{equation*}
  for every $x$.  By Lemma~\ref{asymptot1}, we may shrink the domain further, and ensure that
  \begin{equation*}
    \lim_{x \to 0} \frac{f(x)}{C_0 x^q} = C_1
  \end{equation*}
  for some constant $C_1 \in \Oo^\times$.  Then
  \begin{equation*}
    \lim_{x \to 0} \frac{f(x)}{C_0 C_1 x^q} = 1. \qedhere
  \end{equation*}
\end{proof}

\section{Generic differentiability}

\subsection{Review of strict differentiability} \label{review}
Let $K$ be a topological field (Hausdorff, non-discrete).  The
following definition should be standard, or equivalent to the standard
definition.
\begin{definition} \label{difdef}
  Let $U \subseteq K^n$ be an open set and $f : U \to K^m$ be a
  function.  Then $f$ is \emph{strictly differentiable} at $\ba \in U$
  if for every $i \le n$, the limit
  \begin{equation*}
    g_i(\ba) = \lim_{\substack{\bx \to \ba \\ \epsilon \to 0}}
    \frac{f(\bx + \epsilon \bar{e}_i) - f(\bx)}{\epsilon}
  \end{equation*}
  exists, where $\bar{e}_i$ is the $i$th standard basis vector.  The
  \emph{strict derivative} $Df(\ba)$ is the $m \times n$ matrix whose
  $i$th column is $g_i(\ba)$.
\end{definition}
\begin{example}
  A one-variable function $f : K \to K$ is strictly differentiable at
  $a \in K$ if the following limit exists:
  \begin{equation*}
    Df(a) = \lim_{\substack{x \to a \\ \epsilon \to 0}} \frac{f(x + \epsilon) - f(x)}{\epsilon} = \lim_{(x,y) \to (a,a)} \frac{f(y) - f(x)}{y-x}.
  \end{equation*}
\end{example}
\begin{fact} \phantomsection \label{strict-facts}
  \begin{enumerate}
  \item \label{sf1} A matrix $\mu$ is the strict derivative of $f$ at
    $\ba$ if and only if the following condition holds: for any
    neighborhood $W$ of $\mu$ there is a neighborhood $V$ of $\ba$
    such that if $\bx, \by \in V$, then
    \begin{equation*}
      f(\by) - f(\bx) = \mu' \cdot (\by - \bx)
    \end{equation*}
    for some matrix $\mu' \in W$.
  \item \label{sf2} If $f$ is strictly differentiable on $U$, then the
    strict derivative $Df$ is continuous on $U$.
  \item \label{sf3} A function $f : U \to K^m$ is strictly
    differentiable at $\ba \in U$ if and only if each of the component
    functions $f_1,\ldots,f_m : U \to K$ is strictly differentiable.
    In this case, the strict derivative of $f$ is the matrix whose
    $i$th row is the strict derivative of $f_i$.
  \item \label{sf3.5} If $f$ is strictly differentiable at $\ba$, then
    $f$ is continuous at $\ba$.
  \item \label{sf4} If $f$ is strictly differentiable at $\ba$ and
    $Df(\ba)$ has trivial kernel (i.e., $Df(\ba)$ has rank $n$), then
    $f$ is injective on a neighborhood of $\ba$.
  \item \label{sf5} If $f$ is strictly differentiable at $\ba$ and $g$
    is strictly differentiable at $\bb = f(\ba)$, then $g \circ f$ is
    strictly differentiable at $\ba$, and $D(g \circ f)(\ba) = Dg(\bb)
    \cdot Df(\ba)$.
  \item \label{sf6} If $f : U \to V$ is a homeomorphism between two
    open sets $U,V \subseteq K^n$, and $f$ is strictly differentiable
    at $\ba \in U$, and $Df(\ba)$ is invertible, then $f^{-1}$ is
    strictly differentiable at $\bb = f(\ba)$, and $Df^{-1}(\bb) =
    Df(\ba)^{-1}$.
  \end{enumerate}
\end{fact}
\begin{proof}[Proof sketch]
  Part (\ref{sf3}) is straightforward.  For part (\ref{sf2}), given any $i \le
  n$ and $\ba \in U$ and closed neighborhood $N_0 \ni g_i(a)$, there is a
  neighborhood $N$ of $\ba$ and a neighborhood $N_1 \ni 0$ such that
  \begin{equation*}
    \bx \in N \text{ and } \epsilon \in N_1 \implies
    \frac{f(\bx + \epsilon \bar{e}_i) - f(\bx)}{\epsilon} \in N_0.
  \end{equation*}
  Taking the limit $\epsilon \to 0$, it follows that
  \begin{equation*}
    \bx \in N \implies g_i(\bx) \in N_0,
  \end{equation*}
  and we have shown that $g_i$ is continuous at $\ba$.

  The other points can be seen most easily through non-standard
  analysis.  Say a non-standard element is ``infinitesimal'' if it is
  contained in every standard neighborhood of 0, and let $\bx \approx
  \by$ mean that $\bx - \by$ is a tuple of infinitesimals.  Say that $x$ is $o(y_1,\ldots,y_n)$ if $x$ has the form $\sum_i \epsilon_i y_i$ for some infinitesimals $\epsilon_i$, and say that $\bx$ is $o(y_1,\ldots,y_n)$ if each coordinate of $x$ is $o(y_1,\ldots,y_n)$.\footnote{In a normed field, our definition of ``$x$ is $o(y_1,\ldots,y_n)$'' means precisely that $|x|/\max(|y_1|,\ldots,|y_n|)$ is infinitesimal, as one would expect from the little-o notation.}

  In non-standard terms, a matrix $\mu$ is the strict
  derivative of $f$ at $\ba$ if the following equivalent conditions
  hold:
  \begin{enumerate}[label=(\emph{\alph*})]
  \item \label{alph1} If $i \le n$ and $\bx \approx \ba$ and $\epsilon
    \approx 0$, then
    \begin{equation*}
      \frac{f(\bx + \epsilon \bar{e}_i) - f(\bx)}{\epsilon} \approx \mu \bar{e}_i
    \end{equation*}
  \item \label{alph2} If $i \le n$ and $\bx \approx \ba$ and $\epsilon
    \approx 0$, then
    \begin{equation*}
      f(\bx + \epsilon \bar{e}_i) = f(\bx) + \mu \epsilon \bar{e}_i +
      o(\epsilon)
    \end{equation*}
  \item \label{alph3} If $\bx \approx \ba$ and $\bar{\epsilon} \approx \bar{0}$,
    then
    \begin{equation*}
      f(\bx + \bar{\epsilon}) = f(\bx) + \mu \bar{\epsilon} +
      o(\bar{\epsilon})
    \end{equation*}
    where $o(\bar{\epsilon}) :=
      o(\epsilon_1) + o(\epsilon_2) + \cdots + o(\epsilon_n)$.
  \item \label{alph4} If $\bx \approx \ba$ and $\bar{\epsilon} \approx \bar{0}$,
    then there is a matrix of infinitesimals $\mu_\epsilon$ such that
    \begin{equation*}
      f(\bx + \bar{\epsilon}) = f(\bx) + \mu \bar{\epsilon} +
      \mu_\epsilon \bar{\epsilon}.
    \end{equation*}
  \item \label{alph5} If $\bx \approx \by \approx \ba$, then there is a matrix $\mu'
    \approx \mu$ such that
    \begin{equation*}
      f(\by) = f(\bx) + \mu' (\by - \bx).
    \end{equation*}
  \end{enumerate}
  Condition \ref{alph1} is a non-standard reformulation of
  Definition~\ref{difdef}, and condition \ref{alph5} is a
  non-standard reformulation of the condition in part (\ref{sf1}).
  The implications
  \begin{equation*}
    \ref{alph1}\iff\ref{alph2}\impliedby\ref{alph3}\iff\ref{alph4}\iff\ref{alph5}
  \end{equation*}
  are straightforward, and $\ref{alph2}$ implies $\ref{alph3}$ by
  changing coordinates one-by-one.  This verifies part (\ref{sf1}).

  For part (\ref{sf3.5}), suppose $\mu$ is the strict derivative of
  $f$ at $\ba$.  If $\bx \approx \ba$ then condition~\ref{alph3} gives
  \begin{equation*}
    f(\bx) = f(\ba) + \mu(\bx - \ba) + o(\bx - \ba) \approx f(\ba)
  \end{equation*}
  since $\bx - \ba$ is infinitesimal.  Thus $f$ is continuous at
  $\ba$.
  
  Part (\ref{sf4}) holds because if a standard matrix $\mu$ has
  trivial kernel, then so does every non-standard $\mu' \approx \mu$.
  Then \[\by - \bx \ne \bar{0} \implies f(\by) - f(\bx) = \mu'(\by -
  \bx) \ne \bar{0}\] where $\mu'$ is the matrix from condition \ref{alph5}.

  For part (\ref{sf5}), suppose $\mu_1 = Df(\ba)$ and $\mu_2 =
  Dg(\bb)$.  If $\bx \approx \by \approx \ba$, then $f(\bx) \approx
  f(\by) \approx f(\ba) = \bb$ by continuity, and so
  \begin{equation*}
    g(f(\by)) - g(f(\bx)) = \mu_2' (f(\by) - f(\bx)) = \mu_2' \mu_1'
    (\by - \bx)
  \end{equation*}
  for some matrices $\mu_2' \approx \mu_2$ and $\mu_1' \approx \mu_1$ by condition \ref{alph5}.
  Then $\mu_2' \mu_1' \approx \mu_2 \mu_1$, proving (\ref{sf5}).

  The proof of (\ref{sf6}) is similar: if $\mu = Df(\ba)$ and $\bx
  \approx \by \approx f(\ba)$, then $f^{-1}(\by) \approx f^{-1}(\bx) \approx \ba$ because $f$ is a homeomorphism, and so
  \begin{equation*}
    \by - \bx = \mu' (f^{-1}(\by) - f^{-1}(\bx)) \text{ for some } \mu' \approx \mu
  \end{equation*}
  by condition~\ref{alph5}.
  Then
  \begin{equation*}
    f^{-1}(\by) - f^{-1}(\bx) = (\mu')^{-1} (\by - \bx) \text{ and } (\mu')^{-1} \approx \mu^{-1}. \qedhere
  \end{equation*}
\end{proof}

\begin{fact} \label{val-strict}
  If $(K,v)$ is a valued field, then a matrix $\mu$ is the strict
  derivative $Df(\ba)$ if and only if the following holds: for any
  $\gamma \in \Gamma_K$, there is a neighborhood $N \ni \ba$ such that
  \begin{equation*}
    \bx, \by \in N \implies v(f(\bx) - f(\by) - \mu(\bx - \by)) >
    \gamma + v(\bx - \by).
  \end{equation*}
\end{fact}
\begin{proof}[Proof sketch]
  In non-standard terms, this condition says
  \begin{equation*}
    \bx \approx \by \approx \ba \implies v(f(\bx) - f(\by) - \mu(\bx - \by)) \gg v(\bx - \by),
  \end{equation*}
  where $\gamma_1 \gg \gamma_2$ means that $\gamma_1-\gamma_2$ is
  greater than every standard element of the value group.  Equivalently,
  \begin{equation*}
    (\bx \approx \ba \text{ and } \bar{\epsilon} \approx \bar{0})
    \implies v(f(\bx + \bar{\epsilon}) - f(\bx) - \mu \bar{\epsilon})
    \gg v(\bar{\epsilon}).
  \end{equation*}
  But $v(\bb) \gg v(\bc)$ iff $\bb$ is $o(\bc)$, so we can rephrase
  this as
  \begin{equation*}
    (\bx \approx \ba \text{ and } \bar{\epsilon} \approx \bar{0})
    \implies \left(f(\bx + \bar{\epsilon}) - f(\bx) - \mu
    \bar{\epsilon} \text{ is } o(\bar{\epsilon})\right).
  \end{equation*}
  This is Condition~\ref{alph3} in the proof of Fact~\ref{strict-facts}.
\end{proof}
For reference, here is the definition of multi-variable non-strict
differentiability:
\begin{definition}
  A function $f : U \to \Mm^m$ is \emph{differentiable} at $\ba \in U$
  if there is an $m \times n$ matrix $\mu$ such that for every
  neighborhood $W$ of $\mu$, there is a neighborhood $V$ of $\ba$ such
  that if $\bx \in V$, then
  \begin{equation*}
    f(\bx) - f(\ba) = \mu' \cdot (\bx - \ba)
  \end{equation*}
  for some matrix $\mu' \in W$.  The matrix $\mu$ is called the
  \emph{derivative} of $f$ at $\ba$.
\end{definition}
In non-standard terms, this says that if $\bx \approx \ba$,
then \[f(\bx) = f(\ba) + \mu \cdot (\bx - \ba) + o(\bx - \ba).\]
The derivative is unique, when it exists.  Strict differentiability
implies differentiability, and the strict derivative equals the
derivative when it exists.

\subsection{Generic (strict) differentiability in one variable}
Return to the $P$-minimal monster model $\Mm$.
Let $U \subseteq \Mm^1$ be a definable non-empty open set and let $f : U
\to \Mm$ be a definable function.
\begin{definition}
  If $m \ge 1$, $\lambda \in \Mm^\times$, $q \in \Qq$, and $P_m$
  has a $q$th power map, then $f$ is
  \emph{$(\lambda,m,q)$-differentiable} at $a$ if
  \begin{equation*}
    \lim_{\substack{x \to 0 \\ x \in \lambda P_m}} \frac{f(a+x) - f(a)}{(x/\lambda)^q}
  \end{equation*}
  exists, in which case the limit is the
  \emph{$(\lambda,m,q)$-derivative} at $a$.  Similarly, $f$ is
  \emph{strictly $(\lambda,m,q)$-differentiable} at $a$ if
  \begin{equation*}
    \lim_{\substack{(w,x) \to (a,0) \\ x \in \lambda P_m}}
    \frac{f(w+x) - f(w)}{(x/\lambda)^q}
  \end{equation*}
  exists, in which case the limit is the \emph{strict
  $(\lambda,m,q)$-derivative}.  In both definitions, we shorten
  $(\lambda,m,1)$ to $(\lambda,m)$.
\end{definition}
For example, the $(\lambda,m)$-derivative is like a directional
derivative in the direction $\lambda P_m$, and a $(1,1)$-derivative is
an ordinary derivative.  Note that $(\lambda,m,q)$-differentiability
depends only on the multiplicative coset $\lambda P_m$.  For fixed
$m$, there are only finitely many possibilities for $\lambda P_m$
because $P_m$ has finite index.
\begin{remark}\label{stricter}
  If $f : U \to \Mm$ is $(\lambda,m,q)$-differentiable on $U$, then it
  is strictly $(\lambda,m,q)$-differentiable on a smaller open set $U_0
  \subseteq U$, by Proposition~\ref{strengthen} (with $D = \lambda P_m$).
\end{remark}
Say that $f : U \to \Mm$ is \emph{somewhere locally constant} if there
is a ball $B \subseteq U$ such that $f \restriction B$ is constant,
and \emph{nowhere locally constant} otherwise.
\begin{lemma} \label{a3rev}
  Suppose $f : U \to \Mm$ is definable and nowhere locally constant,
  $\lambda \in \Mm^\times$, and $a \in U$.  Then there is some $m$ and
  $q$ such that $f$ is $(1,m,q)$-differentiable at $a$ and the
  $(1,m,q)$-derivative is non-zero.
\end{lemma}
\begin{proof}
  This comes directly from Proposition~\ref{asymptot3}.  In more
  detail, replacing $f(x)$ with $f(x+a) - f(a)$, we may assume that $a
  = f(a) = 0$.  Applying Proposition~\ref{asymptot3}, we get an $m \ge
  1$ and a ball $B_\epsilon(0)$ such that one of two things happens:
  \begin{itemize}
  \item $f$ is identically zero on $P_m \cap B_\epsilon(0)$,
    contradicting the assumption on $f$.
  \item There is a constant $C \in \Mm^\times$ and a rational number
    $q$ such that \[\lim_{\substack{x \to 0 \\ x \in B_\epsilon(0)
      \cap P_m}} \frac{f(x)}{Cx^q} = 1\] Then $C$ is the
    $(1,m,q)$-derivative of $f$ at 0.  \qedhere
  \end{itemize}
\end{proof}
Recall that in $P$-minimal theories, any infinite one-variable definable
set $D \subseteq \Mm^1$ has non-empty interior \cite[Lemma~4.3(ii)]{p-min}.
\begin{lemma} \label{chaos}
  If $f : U \to \Mm$ is definable and nowhere locally constant, and
  $\lambda \in \Mm^\times$, then there is an integer $m$ and a smaller
  non-empty open set $U_0 \subseteq U$ such that $f$ is strictly
  $(\lambda,m)$-differentiable on $U_0$.
\end{lemma}
\begin{proof}
  The case where $f$ is somewhere locally constant is trivial, so
  assume $f$ is nowhere locally constant.  Changing coordinates ($x' =
  x/\lambda$) we reduce to the case where $\lambda = 1$.  Let
  $D_{q,m}$ be the set of $a \in U$ such that $f$ is
  $(1,m,q)$-differentiable at $a$, and the $(1,m,q)$-derivative is
  non-zero.  By Lemma~\ref{a3rev}, every $a \in U$ belongs to some
  $D_{q,m}$.  By saturation, $U$ is covered by finitely many of the
  $D_{q,m}$.  Then one of them must be infinite, and hence have
  non-empty interior.  Shrinking $U$, we may assume that $U \subseteq
  D_{q,m}$.  So, for every $a \in U$, the $(1,m,q)$-derivative exists
  and is non-zero.  By Remark~\ref{stricter}, we can shrink $U$
  further and arrange for $f$ to be strictly $(1,m,q)$-differentiable
  on $U$.  It remains to show that $q = 1$.

  Let $g(a)$ be the strict $(1,m,q)$-derivative of $f$ at $a \in U$:
  \begin{equation*}
    g(a) = \lim_{\substack{(w,x) \to (a,0) \\ x \in P_m}} \frac{f(w+x)
      - f(w)}{x^q}. \tag{$\ast$}
  \end{equation*}
  Remember that we arranged $g(a) \ne 0$ for all $a$.
  If $c \in P_m$, then
  \begin{equation*}
    \lim_{\substack{(w,x) \to (a,0) \\ x \in P_m}} \frac{f(w+cx) - f(w)}{x^q}
    = \lim_{\substack{(w,y) \to (a,0) \\ y \in P_m}} \frac{f(w+y) - f(w)}{y^q/c^q} = c^q g(a) \tag{$\dag$}
  \end{equation*}
  by the change of coordinates $y = cx$.

  Take some positive integer $k > 1$ such that $k \in P_m$.  For
  example, $k = p^n + 1$ works for $n \gg 0$.  Fix some $a \in U$.  By
  ($\ast$),
      \begin{align*}
      \lim_{\substack{x \to 0 \\ x \in P_m}} \frac{f(a + kx) - f(a)}{x^q} &= \lim_{\substack{x \to 0 \\ x \in P_m}} \sum_{i = 0}^{k-1} \frac{f(a+ix+x) - f(a+ix)}{x^q} \\
      &= \sum_{i = 0}^{k-1} \lim_{\substack{x \to 0 \\ x \in P_m}} \frac{f(a+ix+x) - f(a+ix)}{x^q} \\
      &= \sum_{i = 0}^{k-1} g(a) = k \cdot g(a),
    \end{align*}
    because $(a+ix,x) \to (a,0)$ as $x \to 0$.  On the other hand,
    \begin{equation*}
      \lim_{\substack{x \to 0 \\ x \in P_m}} \frac{f(a + kx) - f(a)}{x^q} = k^q \cdot g(a)
    \end{equation*}
    by ($\dag$).  Thus $k = k^q$.  Regardless of how one chooses
    $k^q$, this is impossible unless $q = 1$.

    Then $q = 1$ and $f$ is strictly $(1,m)$-differentiable on $U$.
\end{proof}
\begin{corollary} \label{cormorant}
  Let $f : U \to \Mm$ be definable.
  \begin{enumerate}
  \item For every ball $B \subseteq U$ and $\lambda \in \Mm^\times$,
    there is a smaller ball $B' \subseteq B$ and an $m \ge 1$ such
    that $f$ is strictly $(\lambda,m)$-differentiable on $B'$.
  \item \label{corm2} In the previous point, $m$ can be chosen
    independently of $\lambda$ and $B$.  In other words, there is an
    $m \ge 1$ such that for every ball $B \subseteq U$ and $\lambda
    \in \Mm^\times$, there is a smaller ball $B' \subseteq B$ on which
    $f$ is strictly $(\lambda,m)$-differentiable.
  \end{enumerate}
\end{corollary}
\begin{proof}
  The first point follows by applying Lemma~\ref{chaos} to $f
  \restriction B$.  The second point follows by compactness.
\end{proof}
\begin{lemma} \label{almost-there}
  Let $f : U \to \Mm$ be definable.  Then there is a non-empty open
  $U_0 \subseteq U$ such that $f$ is differentiable on $U_0$.
\end{lemma}
\begin{proof}
  Let $m$ be as in Corollary~\ref{cormorant}(\ref{corm2}).  Let
  $\lambda_1P_m, \ldots, \lambda_kP_m$ enumerate the finitely many
  cosets of $P_m$.  By Corollary~\ref{cormorant}(\ref{corm2}) we can
  shrink $U$ and assume that $f$ is strictly
  $(\lambda_1,m)$-differentiable on $U$.  Repeating this for
  $\lambda_2, \ldots, \lambda_k$ we may assume that $f$ is strictly
  $(\lambda_i,m)$-differentiable for every $i$.  Then $f$ is strictly
  $(\lambda,m)$-differentiable for every $\lambda$.  Let
  $g_\lambda(a)$ denote $\lambda^{-1}$ times the strict $(\lambda,m)$-derivative at $a \in
  U$:
  \begin{equation*}
    g_\lambda(a) = \lim_{\substack{(w,x) \to (a,0) \\ x \in \lambda
        P_m}} \frac{f(w+x) - f(w)}{x}.
  \end{equation*}
  Note that if $a \in U$ and $c \in \lambda P_m$, then
  \begin{equation*}
    \lim_{\substack{(w,x) \to (a,0) \\ x \in P_m}} \frac{f(w+cx) -
      f(w)}{x} = \lim_{\substack{(w,y) \to (a,0) \\ y \in c P_m}}
    \frac{f(w+y) - f(w)}{y/c} = c g_\lambda(a) \tag{$\dag$}
  \end{equation*}
  via the change of coordinates $y = cx$.
  \begin{claim}
    For any $a \in U$ and $\lambda \in \Mm^\times$, $g_\lambda(a) =
    g_1(a)$.
  \end{claim}
  \begin{claimproof}
    Take $b \in \lambda P_m$ so small that $1 + b \in P_m$.  Then
    \begin{equation*}
      \lim_{\substack{x \to 0 \\ x \in P_m}} \frac{f(a+(1+b)x) -
        f(a)}{x} = (1+b)g_1(a)
    \end{equation*}
    by ($\dag$).  But the same value can alternatively be calculated
    as
    \begin{gather*}
      \lim_{\substack{x \to 0 \\ x \in P_m}} \frac{f(a+x+bx) - f(a+x)}{x} +
      \lim_{\substack{x \to 0 \\ x \in P_m}} \frac{f(a+x) - f(a)}{x} \\
      = b g_\lambda(a) + g_1(a)
    \end{gather*}
    by ($\dag$).  Thus
    \begin{equation*}
      (1 + b)g_1(a) = b g_\lambda(a) + g_1(a),
    \end{equation*}
    implying $g_\lambda(a) = g_1(a)$.
  \end{claimproof}
  It follows that
  \begin{equation*}
    \lim_{\substack{x \to 0 \\ x \in \lambda P_m}} \frac{f(a+x) -
      f(a)}{x} = g_\lambda(a) = g_1(a)
  \end{equation*}
  for any coset $\lambda P_m$.  Since there are only finitely many
  cosets,
  \begin{equation*}
    \lim_{x \to 0} \frac{f(a+x) - f(a)}{x} = g_1(a),
  \end{equation*}
  and we see that $g_1(x)$ is the derivative of $f$.  We have shown
  that $f$ is differentiable on $U$.
\end{proof}
\begin{proposition} \label{one-var-diff}
  If $f : U \to \Mm$ is definable, then $f$ is strictly differentiable
  at all but finitely many points.
\end{proposition}
\begin{proof}
  Let $X$ be the set of points on which $f$ is \emph{not} strictly
  differentiable.  If $X$ is infinite, it contains an open set by
  $P$-minimality.  Restricting $f$ to this open set, we may assume that
  $f$ is nowhere strictly differentiable.  Lemma~\ref{almost-there}
  gives a smaller open set on which $f$ is differentiable, and then
  Remark~\ref{stricter} gives a further smaller open set on which $f$
  is strictly differentiable, a contradiction.
\end{proof}

\begin{remark}
  Proposition~\ref{one-var-diff} shows that if $f : \Mm \to \Mm$ is
  definable and $a \in \Mm$ is generic, then $f$ has the following
  asymptotic expansion around $a$:
  \begin{equation*}
    f(a + \epsilon) = f(a) + f'(a)\epsilon + o(\epsilon).
  \end{equation*}
  Now consider the remainder $f(a + \epsilon) - (f(a) +
  f'(a)\epsilon)$.  Using Proposition~\ref{asymptot3}, the remainder must
  look like a power of $\epsilon$:
  \begin{equation*}
    f(a + \epsilon) - f(a) - f'(a)\epsilon \propto C_a \epsilon^q,
    \tag{$\ast$}
  \end{equation*}
  at least for $\epsilon$ in some $P_m$.  Using the proof strategy
  from Lemmas~\ref{chaos} and \ref{almost-there}, one can show that $q
  = 2$, that ($\ast$) holds on \emph{all} cosets of $P_m$, and that
  $C_a = f''(a)/2$.  Thus
  \begin{equation*}
    f(a + \epsilon) = f(a) + f'(a)\epsilon +
    \frac{f''(a)}{2}\epsilon^2 + o(\epsilon^2).
  \end{equation*}
  Continuing on in this way, one should get a Taylor series expansion
  \begin{equation*}
    f(a + \epsilon) = f(a) + f'(a)\epsilon + \cdots +
    \frac{f^{(n)}(a)}{n!}\epsilon^n + o(\epsilon^n)
  \end{equation*}
  for each $n$.  There should also be a multivariable version of these
  statements.  Unfortunately, the calculations are too complicated to
  carry out in this paper.
\end{remark}

\subsection{Multivariable generic differentiability}
\begin{theorem} \label{generic-diff}
  Let $U \subseteq \Mm^n$ be non-empty, open, and definable.  Let $f :
  U \to \Mm^m$ be a definable function.  Then $f$ is strictly
  differentiable on a definable open set $U_0 \subseteq U$ with $\dim
  U \setminus U_0 < \dim U$.
\end{theorem}
\begin{proof}
  We will only consider the case $m=1$, which is sufficient, by
  Fact~\ref{strict-facts}(\ref{sf3}).  By the usual methods, it
  suffices to find a non-empty open subset $U_0 \subseteq U$ on which
  $f$ is strictly differentiable.

  Let $U_i$ be the set of points in $U$ such that the $i$th partial
  derivative exists:
  \begin{equation*}
    \lim_{y \to 0}
    \frac{f(\ba + y \bar{e}_i) -
      f(\bar{a})}{y} \text{ exists}.
  \end{equation*}
  The complement $U \setminus U_i$ cannot contain a ball, by
  Proposition~\ref{one-var-diff}.  Therefore, $U_i$ contains a ball.
  Shrinking $U$, we may assume that the $i$th partial derivative
  exists everywhere on $U$.  By Proposition~\ref{strengthen}, we may further
  shrink $U$ and assume that
  \begin{equation*}
    \lim_{(\bx,y) \to (\ba,0)} \frac{f(\bx + y \bar{e}_i) - f(\bx)}{y}
    \text{ exists}
  \end{equation*}
  for any $\ba \in U$.  Repeating this for $i = 1, 2, \ldots, n$, we
  can arrange this to hold for every $i$.  Then $f$ is strictly
  differentiable (Definition~\ref{difdef}).
\end{proof}

\subsection{The inverse function theorem}
We quickly check that one of the standard proofs of the inverse
function theorem works in our context.
\begin{lemma} \label{contract}
  Let $B \subseteq \Mm^n$ be a ball around 0.  Let $f : B \to B$ be a
  function which is contracting:
  \begin{equation*}
    v(f(x) - f(y)) > v(x-y) \text{ for distinct } x,y \in B.
  \end{equation*}
  \begin{enumerate}
  \item There is a unique $x \in B$ such that $f(x) = x$.
  \item \label{con2} For any $c \in B$, there is a unique $x \in B$
    such that $f(x) = x - c$.
  \end{enumerate}
\end{lemma}
\begin{proof}
  \begin{enumerate}
  \item Uniqueness is clear.  Suppose existence fails ($f$ has no
    fixed point).  Note that $f$ is continuous.  Let $g : B \to
    \Gamma$ be the function $g(x) = v(f(x) - x)$.  Then
    \begin{equation*}
      g(f(x)) = v(f(f(x)) - f(x)) > v(f(x) - x) = g(x),
    \end{equation*}
    so the set $\{g(x) : x \in B\}$ has no maximum.  Since $\Gamma$ is
    definably well-ordered, $\{g(x) : x \in B\}$ has no upper bound.
    Then each of the closed sets
    \begin{equation*}
      D_\gamma = \{x \in B : g(x) \ge \gamma\} = \{x \in B : v(f(x) - x)
      \ge \gamma\}
    \end{equation*}
    is non-empty.  By definable compactness, there is $a \in
    \bigcap_{\gamma \in \Gamma} D_\gamma$.  But then $v(f(a) - a) =
    +\infty$, so $a$ is a fixed point.
  \item Apply the previous point to the function $g : B \to B$ given
    by $g(x) = f(x) + c$.  \qedhere
  \end{enumerate}
\end{proof}
\begin{lemma}
  Let $U$ be a definable neighborhood of $\ba \in \Mm^n$ and let $f :
  U \to \Mm^n$ be a definable function which is strictly
  differentiable on $U$, such that $Df(\ba)$ is invertible.  Then the
  image of $f$ contains a ball around $f(\ba)$.
\end{lemma}
\begin{proof}
  Changing coordinates, we may assume $\ba = f(\ba) = \bar{0}$, and
  $Df(\ba)$ is the $n \times n$ identity matrix.  Let $h : U \to
  \Mm^n$ be the function $h(\bx) = \bx - f(\bx)$.  Then $h$ is
  strictly differentiable on $U$, and $Dh(\bar{0})$ vanishes.  By
  Fact~\ref{val-strict} (with $\gamma = 0$), there is a ball $B \ni
  \bar{0}$ with $B \subseteq U$ such that
  \begin{equation*}
    \bx, \by \in B \implies v(h(\bx) - h(\by)) > v(\bx - \by). \tag{$\ast$}
  \end{equation*}
  In particular, $h$ is contracting.  Note that $h(\bar{0}) =
  \bar{0}$, so if $\bx \in B$, then
  \begin{equation*}
    v(h(\bx)) = v(h(\bx) - h(\bar{0})) > v(\bx - \bar{0}) = v(\bx).
  \end{equation*}
  Thus $h$ maps $B$ into $B$.  By Lemma~\ref{contract}(\ref{con2}),
  for any $\bar{c} \in B$, there is $\bar{x} \in B$ such that
  \begin{equation*}
    \bx - f(\bx) = h(\bx) = \bx - \bc,
  \end{equation*}
  so $f(\bx) = \bc$.  We have shown that $\img(f)$ contains $B$.
\end{proof}
\begin{corollary}
  Let $U \subseteq \Mm^n$ be a definable open set and $f : U \to
  \Mm^n$ be definable and strictly differentiable.  If $Df(\ba)$ is
  invertible for every $\ba \in U$, then $f$ is an open map.
\end{corollary}
\begin{theorem}[Inverse function theorem] \label{ift}
  Let $U \subseteq \Mm^n$ be a definable open set and $f : U \to
  \Mm^n$ be a definable, strictly differentiable function.  Let $\ba
  \in U$ be a point such that the strict derivative $Df(\ba)$ is
  invertible.  Then there is a smaller open set $\ba \in U_0 \subseteq
  U$ such that $f$ restricts to a homeomorphism $U_0 \to V_0$ for some
  neighborhood $V_0 \ni f(\ba)$, and both $f : U_0 \to V_0$ and
  $f^{-1} : V_0 \to U_0$ are strictly differentiable.
\end{theorem}
\begin{proof}
  The derivative $Df(\bx)$ is continuous by
  Fact~\ref{strict-facts}(\ref{sf2}).  So we can find a neighborhood
  $U_0 \ni \ba$ such that $Df(\bx)$ is invertible for every $\bx \in
  U_0$.  The fact that $Df(\ba)$ is invertible implies that $f$ is
  injective on a neighborhood of $\ba$, by
  Fact~\ref{strict-facts}(\ref{sf4}).  Shrinking $U_0$ further, we may
  assume that $f \restriction U_0$ is injective.  Then $f \restriction
  U_0$ is a continuous, injective, open map, so it is a homeomorphism
  onto its image $V_0$.  The inverse function $f^{-1} : V_0 \to U_0$
  is strictly differentiable by Fact~\ref{strict-facts}(\ref{sf6}).
\end{proof}

\section{Recognizing $p$-adic Lie groups} \label{lazard-nonsense}
Recall that a profinite group $G$ is a \emph{pro-$p$ group} if it is
an inverse limit of finite $p$-groups \cite[Proposition~1.12]{app}.
Using work of Lazard \cite{lazard}, as reported in \cite{app}, one can
prove the following:
\begin{fact} \label{reference-hunt}
  Let $G$ be a pro-$p$ group.  Let $S_i$ be the image of the $p^i$th
  power map $G \to G$.  Let $n$ be an integer.  Suppose the following
  conditions hold:
  \begin{enumerate}
  \item $G$ has no $p$-torsion.
  \item $S_1$ and $S_2$ are normal open subgroups of $G$, and the
    quotients $G/S_1$ and $G/S_2$ are abelian.
  \item $G/S_1$ has size $p^n$.
  \end{enumerate}
  Then $G$ is isomorphic, as a topological group, to an
  $n$-dimensional Lie group over $\Qq_p$.
\end{fact}
Fact~\ref{reference-hunt} is probably obvious or well-known to experts
on $p$-adic Lie groups, but for the rest of us, we need to chain
together some references from \cite{app}.  First we recall some
definitions and facts:
\begin{enumerate}
\item \label{item1} A pro-$p$ group $G$ is \emph{powerful}
  \cite[Definition~3.1]{app} if $p$ is odd and $G/\overline{G^p}$ is
  abelian, or $p = 2$ and $G/\overline{G^4}$ is abelian, where $G^n$
  denotes the subgroup generated by $n$th powers.
\item \label{item2} If $G$ is a profinite group, the \emph{Frattini subgroup}
  $\Phi(G)$ is the intersection of all maximal open proper subgroups
  of $G$ \cite[Definition~1.8]{app}.
\item \label{item3} If $G$ is a pro-$p$ group, then the Frattini subgroup $\Phi(G)$
  is the closure of the subgroup generated by $p$th powers and
  commutators \cite[Proposition~1.13]{app}:
  \begin{equation*}
    \Phi(G) = \overline{G^p [G,G]}.
  \end{equation*}
\item \label{item4} A \emph{topological generating set} of $G$ is a subset $X$
  generating a dense subgroup of $G$, and $G$ is \emph{finitely
  generated} if there is a finite topological generating set
  \cite[p.~20]{app}.
\item \label{item5} A pro-$p$ group is finitely generated iff $\Phi(G)$ is open
  \cite[Proposition~1.14]{app}.
\item \label{item6} If $G$ is a finitely generated powerful pro-$p$ group, then
  $\Phi(G)$ is exactly the set of $p$th powers \cite[Lemma~3.4]{app}.
\item \label{item7} If $G$ is a topological group, then $\operatorname{d}(G)$ is the
  minimum cardinality of a topological generating set for $G$.  When
  $G$ is a pro-$p$ group, $\operatorname{d}(G)$ equals the dimension
  of $G/\Phi(G)$ as a vector space over $\Ff_p$.  (See the paragraph
  above \cite[Theorem~3.8]{app}.)
\item \label{item8} A pro-$p$ group is \emph{uniformly powerful} if it is powerful,
  finitely generated, and has no $p$-torsion
  \cite[Theorem~4.5]{app}.\footnote{Theorem~4.5 in \cite{app} says
  ``torsion-free'' rather than ``$p$-torsion free'', but as noted in
  the third sentence of the proof, these conditions are equivalent for
  pro-$p$ groups.}
\item \label{item9} If $G$ is a uniformly powerful pro-$p$ group, then $G$ is a Lie
  group over $\Qq_p$ \cite[Theorem~8.18]{app}.  Moreover, the
  dimension of $G$ as a Lie group over $\Qq_p$ equals
  $\operatorname{d}(G)$ by Theorem~8.36 and Definition~4.7 of
  \cite{app}.
\end{enumerate}
Combining these nine ingredients, we prove Fact~\ref{reference-hunt}:
\begin{proof}
  By our assumptions on $G$,
  \begin{gather*}
    \overline{G^p} = \overline{\langle S_1 \rangle} = \overline{S_1} =
    S_1 \text{ for any $p$} \\
    \overline{G^4} = \overline{\langle S_2 \rangle} = \overline{S_2} =
    S_2 \text{ when $p = 2$},
  \end{gather*}
  and both the quotients $G/S_1$ and $G/S_2$ are abelian.  Thus $G$ is
  a powerful pro-$p$ group (Ingredient \ref{item1}).  The Frattini subgroup  $\Phi(G)$ (Ingredient \ref{item2}) is clopen
  because it contains (Ingredient \ref{item3}) the clopen subgroup $S_1$:
  \begin{equation*}
    \Phi(G) = \overline{G^p [G,G]} \supseteq \overline{G^p} = S_1.
  \end{equation*}
  Therefore (Ingredient \ref{item5}) $G$ is finitely generated (Ingredient \ref{item4}), which implies (Ingredient \ref{item6}) that $\Phi(G) =
  \{x^p : x \in G\} = S_1$.  Moreover, $G$ is uniformly powerful
  because $G$ has no $p$-torsion (Ingredient \ref{item8}).  Then $G$ is a $p$-adic Lie group (Ingredient \ref{item9}) of
  dimension
  \begin{equation*}
    \dim(G) = \operatorname{d}(G) = \dim_{\Ff_p} G/\Phi(G) =
    \dim_{\Ff_p} G/S_1 = n,
  \end{equation*}
  (Ingredients \ref{item9}, \ref{item7})
  since $G/S_1$ has cardinality $p^n$.
\end{proof}
We will apply this to the following setting:
\begin{lemma} \label{recognizer2}
  Let $K$ be a degree $e$ extension of $\Qq_p$, and let $\Oo_K$ be the
  ring of integers in $K$.  Let $d$ be an integer.  Let $\star$ be a
  group operation on $\Oo_K^d$.  Let $f : \Oo_K^d \to \Oo_K^d$ be the
  $p$th power map (with respect to $\star$), i.e.,
  \begin{equation*}
    f(x) = \underbrace{x \star x \star \cdots \star x}_{\text{$p$ times}}.
  \end{equation*}
  Suppose the following conditions hold:
  \begin{enumerate}
  \item \label{sump1} $(\Oo_K^d,\star)$ is a topological group with respect to the
    usual topology on $\Oo_K$.
  \item \label{sump2} The identity element is $\bar{0}$.
  \item \label{sump3} For every $n \in \Nn$, the set $p^n \Oo_K^d$ is a normal
    subgroup of $(\Oo_K^d,\star)$ whose cosets are the additive cosets
    $\ba + p^n \Oo_K^d$.  Equivalently, the equivalence relation
    \begin{equation*}
      \bx \equiv_n \by \iff \bx - \by \in p^n \Oo_K^d
    \end{equation*}
    is a congruence on the group $(\Oo_K^d,\star)$.
  \item \label{sump4} For $n \le 2$, the group structure on
    $(\Oo_K^d,\star)/p^n\Oo_K^d$ agrees with the additive structure
    $(\Oo_K^d,+)/p^n\Oo_K^d$.  In other words,
    \begin{equation*}
      \bx + \by \equiv \bx \star \by \pmod{p^n \Oo_K^d}.
    \end{equation*}
  \item \label{sump5} The map $f$ scales distances by a factor of $p$:
    \begin{equation*}
      v(f(\bx) - f(\by)) = v(p) + v(\bx - \by).
    \end{equation*}
  \item \label{sump6} The image of $f$ contains $p \Oo_K^d$.
  \end{enumerate}
  Then $(\Oo_K,\star)$ is an $n$-dimensional $p$-adic Lie group for $n
  = de$.
\end{lemma}
\begin{proof}
  Let $G = (\Oo_K^d,\star)$ and let $G_n$ be the normal subgroup $(p^n
  \Oo_K^d, \star)$ from Assumption~\ref{sump3}.  Since we are taking
  the standard topology on $\Oo_K^d$ (Assumption~\ref{sump1}), the
  descending chain $G \supseteq G_1 \supseteq G_2 \supseteq \cdots$ is
  a neighborhood basis of the identity element $\bar{0}$
  (Assumption~\ref{sump2}).  By Assumption~\ref{sump3}, the index $|G
  : G_i|$ equals the additive index $|(G,+) : (G_i,+)|$.  Since
  $(\Oo_K^d,+)$ is isomorphic to $(\Zz_p^{de},+)$, the index $|G :
  G_i|$ equals $p^{dei}$.  In particular,
  \begin{equation*}
    G \cong \varprojlim_{i \to \infty} G/G_i \text{ is a pro-$p$ group}.
  \end{equation*}
  Taking $\by = \bar{0}$ in Assumption~\ref{sump5}, we see that
  \begin{equation*}
    v(f(\bx)) = v(\bx) + v(p).  \tag{$\ast$}
  \end{equation*}
  In particular, if $\bx \ne \bar{0}$, then $v(f(\bx)) = v(\bx) + v(p)
  < \infty$, so $f(\bx) \ne \bar{0}$.  This shows that the group
  $(\Oo_K^d,\star)$ has no $p$-torsion.

  Equation ($\ast$) also shows that $f$ maps $p^i \Oo_K^d$ into
  $p^{i+1} \Oo_K^d$.  In fact,
  \begin{equation*}
    f(p^i \Oo_K^d) = p^{i+1} \Oo_K^d
  \end{equation*}
  because if $\by$ is on the right hand side, then $\by \in p^{i+1}
  \Oo_K^d \subseteq p \Oo_K^d \subseteq \img(f)$ by
  Assumption~\ref{sump6}, so $\by = f(\bx)$ for \emph{some} $\bx \in
  \Oo_K^d$.  But then Equation~($\ast$) shows that $v(\bx) = v(\by) -
  v(p) \ge i \cdot v(p)$, so $\bx \in p^i \Oo_K^d$.

  Consequently, the image of the $p^i$th power map is exactly $p^i
  \Oo_K^d = G_i$.  In the notation of Fact~\ref{reference-hunt}, we
  have $S_i = G_i$.  In particular, $S_i$ is a normal open subgroup of
  $G$, and $G/S_i$ has size $p^{dei}$.  Then $G/S_1$ has size
  $p^{de}$.  For $i \le 2$, the quotient $G/S_i$ is abelian by
  Assumption~\ref{sump4}.

  Now all the conditions of Fact~\ref{reference-hunt} are satisfied,
  and so $G$ is a $p$-adic Lie group of dimension $de$.
\end{proof}

\section{Conditions $\mathfrak{A}$, $\mathfrak{B}$, $\mathfrak{C}$, $\mathfrak{D}$, $\mathfrak{E}$}
Work in a model $M$ of the $P$-minimal theory $T$.  Let $\Oo$ denote the valuation ring of $M$.
\begin{definition} \label{ABC}
  Let $G = (\Oo^d,\star)$ be a definable group with underlying set
  $\Oo^d$, such that $\bar{0} \in \Oo^d$ is the identity element.  We
  define the following conditions:
  \begin{itemize}
  \item[$(\mathfrak{A}_n)$] For $n \in \Nn$, the group $G$ satisfies
    condition $\mathfrak{A}_n$ if the ball $p^n\Oo^d$ is a normal
    subgroup of $G$, and the $\star$-cosets of $p^n\Oo^d$ agree with
    the additive cosets:
    \begin{equation*}
      \{a \star p^n\Oo^d : a \in \Oo^d\} = \{a + p^n\Oo^d : a \in \Oo^d\}.
    \end{equation*}
    Equivalently, the relation $v(\bx - \by) \ge n \cdot v(p)$ is a
    congruence on $G$.
  \item[$(\mathfrak{A}_\omega)$] The group $G$ satisfies
    $\mathfrak{A}_\omega$ if it satisfies $\mathfrak{A}_n$ for all $n
    \in \Nn$.
  \item[$(\mathfrak{A}_\infty)$] The group $G$ satisfies
    $\mathfrak{A}_\infty$ if for any $\bx, \by, \bz$ in $\Oo^n$ and
    $\gamma \in \Gamma$, the relation $v(\bx - \by) \ge \gamma$ if a
    congruence on $G$.
  \item[$(\mathfrak{B}_n)$] For $n \in \Nn$, the group $G$ satisfies
    condition $\mathfrak{B}_n$ if it satisfies $\mathfrak{A}_n$, and
    moreover the group operation on the $\star$-cosets of $p^n\Oo^d$
    is the usual addition.

    Equivalently, $G = (\Oo^d,\star)$ satisfies $\mathfrak{B}_n$ if
    \begin{equation*}
      \bx \star \by \equiv \bx + \by \pmod{p^n\Oo^d}
    \end{equation*}
    for every $\bx,\by \in \Oo^d$.
  \item[$(\mathfrak{B}_\omega)$] The group $G$ satisfies
    $\mathfrak{B}_\omega$ if it satisfies $\mathfrak{B}_n$ for all $n
    \in \Nn$.
  \item[$(\mathfrak{C}_n)$] For $n \in \Nn$, the group $G$ satisfies
    condition $\mathfrak{C}_n$ if
    \begin{itemize}
    \item It satisfies $\mathfrak{A}_n$
    \item Let $f(\bx,\by) = \bx \star \by$.  For any multi-indices
      $I,J$ with $|I|+|J| \le n$, the Taylor series coefficient
      \begin{equation*}
        \bar{c}_{I,J} = \frac{1}{I!J!} \frac{\partial f}{\partial
          \bx^I \partial \by^J} (\bar{0}, \bar{0})
      \end{equation*}
      exists and is in $\Oo^d \cap p^{|I|+|J|-1}\Oo^d$.
    \item The quotient group $G/p^n\Oo^d$ has the following group
      structure:
      \begin{equation*}
        \bx \star \by \equiv \sum_{\substack{I,J \\ |I|+|J| \le n}} \bar{c}_{I,J}
        \bx^I \by^J \pmod{p^n\Oo^d}.
      \end{equation*}
    \end{itemize}
  \item[$(\mathfrak{C}_\omega)$] $G$ satisfies $\mathfrak{C}_\omega$
    if it satisfies $\mathfrak{C}_n$ for all $n \in \Nn$.
  \item[$(\mathfrak{D})$] $G$ satisfies $\mathfrak{D}$ if the $p$th
    power map $f : G \to G$ satisfies the condition
    \begin{equation*}
      v(f(\bx) - f(\by)) = v(\bx - \by) + v(p)
    \end{equation*}
    for any distinct $\bx, \by \in G$.
  \item[$(\mathfrak{E})$] $G$ satisfies $\mathfrak{E}$ if the image of
    the $p$th power map $f : G \to G$ contains the ball $p \Oo^d$.
  \end{itemize}
\end{definition}
\begin{remark}
  Condition $\mathfrak{B}_\omega$ says something like $x \star y
  \approx x + y$, i.e., $\star$ is close to addition.  We could define
  $\mathfrak{B}_\infty$ by analogy to $\mathfrak{A}_\infty$, but it
  would imply $x \star y = x + y$ for all $x,y$.
\end{remark}
\begin{remark} \label{def-el}
  Conditions $\mathfrak{A}_n$, $\mathfrak{A}_\infty$,
  $\mathfrak{B}_n$, $\mathfrak{C}_n$, $\mathfrak{D}$, and
  $\mathfrak{E}$ are definable in families, and preserved in
  elementary extensions.  Conditions $\mathfrak{A}_\omega$,
  $\mathfrak{B}_\omega$, and $\mathfrak{C}_\omega$ are type-definable
  in families, and preserved in elementary extensions.
\end{remark}
\begin{example} \label{c-source}
  In the standard model $K$, suppose that $(\Oo^n,\star)$ is a
  definable group structure with $\bar{0}$ as the identity element,
  such that $\star$ is given by a formal power series
  \begin{equation*}
    \bx \star \by = \sum_{I,J} \bc_{I,J} \bx^I \by^J.
  \end{equation*}
  Since $\bx \star \bar{0} = \bar{0} \star \bx = \bx$, we in fact have
  \begin{equation*}
    \bx \star \by = \bx + \by + \sum_{\substack{I,J \\ |I| \ge 1, ~ |J| \ge 1}}
    \bc_{I,J} \bx^I \by^J.
  \end{equation*}
  Suppose $\bc_{I,J} \in p^{|I|+|J|-1} \Oo^n$ for each $I,J$ with
  $|I|,|J| \ge 1$.  Then $\bc_{I,J} \in \Oo^n \cap p^{|I| + |J| - 1}
  \Oo^n$ for each $I,J$, and condition $\mathfrak{C}_\omega$ is easy
  to verify.  The fact that $\mathfrak{A}_n$ holds for all $n$ also
  implies that $\mathfrak{A}_\infty$ holds, since we are in the
  standard model.
\end{example}

\subsection{From $\mathfrak{A}_\omega$ to compact domination}
First, we recall a few well-known facts about $P$-minimal theories.  Recall that \emph{dp-minimal} theories are theories of dp-rank 1 \cite[Definition~4.27]{NIPguide}, \cite{dpExamples}.  Informally, these are the ``1-dimensional'' NIP theories.
\begin{remark}\label{pdpm}
  $P$-minimal theories are dp-minimal, because $p$-adically closed
  fields are dp-minimal \cite[\S 6]{dpExamples}, and dp-minimality
  depends only on the collection of definable sets in one variable.
\end{remark}
Recall that \emph{distal theories} are a special class of NIP theories that are ``anti-stable'', in some informal sense.  See \cite[Chapter~9]{NIPguide} for a precise definition.
\begin{remark}
  $P$-minimal theories are distal.  This is well-known, but I don't have
  a reference on hand, so here is a proof:
\end{remark}
\begin{proof}
  $p$-adically closed fields have definable Skolem functions, so there
  is a definable function $f_0(x,y)$ such that
  \begin{equation*}
    f_0(x+y,xy) \in \{x,y\}.
  \end{equation*}
  Letting $f(x,y) = f_0(x+y,xy)$, we get a definable function $f$ such
  that
  \begin{gather*}
    f(x,y) = f(y,x) \\
    f(x,y) \in \{x,y\}.
  \end{gather*}
  Using $f$, we see that $\tp(a,b) \ne \tp(b,a)$ for any $a \ne b$.
  Consequently, the only totally indiscernible sequences are the
  constant sequences.  The function $f$ also exists in $P$-minimal
  expansions, of course.  Therefore, $P$-minimal theories also have the
  property that totally indiscernible sequences are constant.  If $p$
  is a global type which is generically stable, then the Morley
  sequence of $p$ is totally indiscernible
  \cite[Proposition~3.2(ii)]{udi-anand}, hence constant, which makes
  $p$ be a constant/realized type.  Thus $P$-minimal theories have no
  non-constant generically stable types.  By
  \cite[Corollary~9.19]{NIPguide}, $P$-minimal theories are distal.
\end{proof}
Recall that a definable group $G$ has \emph{finitely satisfiable
generics} (fsg) if there is a global type $p$ on $G$ and a small set
$A_0$ such that every translate of $p$ is finitely satisfiable in
$A_0$ \cite[Definition~4.1]{HPP}.  Equivalently, $G$ has fsg if there
is a gobal type $p$, finitely satisfiable in a small set $A_0$, such
that $p$ is ``almost translation-invariant'' in the sense that $\{g
\cdot p : g \in G\}$ is small.
\begin{remark} \label{fsg}
  If $G$ is a definable group in a distal theory, then the following
  are equivalent:
  \begin{enumerate}
  \item $G$ has finitely satisfiable generics (fsg).
  \item $G$ is compactly dominated.
  \item $G$ is compactly dominated by (normalized) Haar measure on $G/G^{00}$.
  \end{enumerate}
  The equivalence of (2) and (3) is \cite[Lemma~8.36]{NIPguide}.  The
  equivalence of (1) and (2) is stated in
  \cite[Example~8.42]{NIPguide}, but the proof is a bit hidden.  For
  reference:
  \begin{itemize}
  \item $G$ is compactly dominated iff $G$ has a smooth left-invariant
    measure \cite[Theorem~8.37]{NIPguide}.
  \item $G$ is fsg iff $G$ is definably amenable and $G$ has a
    generically stable left-invariant measure
    \cite[Proposition~8.33]{NIPguide}.
  \item $G$ is definably amenable iff $G$ has a left-invariant measure
    \cite[Definition 8.12]{NIPguide}.  Thus $G$ is fsg iff $G$ has a
    generically stable left-invariant measure.
  \item Smooth measures are generically stable \cite[Lemma~7.17 and
    \S7.5]{NIPguide}
  \item In distal theories, generically stable measures are smooth
    \cite[Proposition~9.26]{NIPguide}.
  \item Therefore, in distal theories, generically stable measures
    are the same thing as smooth measures, and fsg groups are the same
    thing as compactly dominated groups.
  \end{itemize}
\end{remark}
Now work in a monster model $\Mm$ of the $P$-minimal theory $T$.
\begin{proposition} \label{a-b}
  Suppose $G = (\Oo^d,\star)$ satisfies condition
  $\mathfrak{A}_\omega$.
  \begin{enumerate}
  \item $G$ is fsg and compactly dominated.
  \item $G^{00}$ is the subgroup $\bigcap_{n = 0}^\infty p^n \Oo^d$.
  \end{enumerate}
\end{proposition}
\begin{proof}
  The proof requires several steps.  We first need to
  analyze the groups $(\Oo,+)$ and $(\Oo^d,+)$.

  In the base theory $\Th(K)$ of $p$-adically closed fields (i.e.,
  the reduct to $\mathcal{L}_{Rings}$) the additive group $(\Oo,+)$ is
  fsg and the connected component $(\Oo,+)^{00}$ is $\bigcap_{n = 0}^\infty p^n \Oo$, by \cite[Corollaries~2.3--2.4]{O-P}.

  The group $(\Oo,+)$ is also fsg in the $P$-minimal expansion $T$.
  To see this, take an almost translation-invariant global type $q$ in
  the $\mathcal{L}_{Rings}$-theory, finitely satisfiable in a small
  set $A$.  Then $q$ extends uniquely to a global type $\hat{q}$ in
  the expansion, because the boolean algebra of definable sets is the
  same in the two languages (by $P$-minimality), so the spaces of
  global 1-types are the same.  It is clear that $\hat{q}$ is almost
  translation-invariant and finitely satisfiable in $A$, and so
  $(\Oo,+)$ is fsg in the expansion.

  Similarly, $(\Oo,+)^{00}$ in the expansion is still $\bigcap_{n=0}^\infty p^n \Oo$.  This holds because the collection of
  type-definable subsets of $(\Oo,+)$ is the same in both languages,
  by $P$-minimality again.  So the two collections
  \begin{align*}
    \{H : H \text{ is a type-definable, } H \lhd (\Oo,+), \text{ and
      $(\Oo,+)/H$ is small}\} &\text{ in } \Th(K)
    \\
        \{H : H \text{ is a type-definable, } H \lhd (\Oo,+), \text{ and
      $(\Oo,+)/H$ is small}\} &\text{ in } T
  \end{align*}
  are identical, and have the same minimum element.

  For the rest of the proof, we remain in the $P$-minimal expansion $T$,
  rather than the base theory $\Th(K)$.

  An extension of an fsg group by an fsg group is fsg
  \cite[Proposition~4.5]{HPP}, so in particular a product of two fsg
  groups is fsg.  Therefore $(\Oo^d,+)$ is fsg (in the $P$-minimal
  expansion $T$).  Moreover, $(G \times H)^{00} = G^{00} \times
  H^{00}$.  Therefore \[(\Oo^d,+)^{00} = \bigcap_{n =0}^\infty
  p^n\Oo^d\] holds (in the $P$-minimal expansion $T$).  Note that
  $(\Oo^d,+)/(\Oo^d,+)^{00} \cong \Oo_K^d$.  (For example, it's
  $\Zz_p^d$ when $\Mm$ is elementarily equivalent to $\Qq_p$.)  By
  Remark~\ref{fsg}, the group $(\Oo^d,+)^{00}$ is compactly dominated
  by Haar measure on $\Oo_K^d$.

  The fact that $(\Oo^d,+)$ is fsg and compactly dominated implies by
  \cite[Propositions~8.32, 8.38]{NIPguide} that there is a unique
  smooth, translation-invariant measure $\mu$ on $(\Oo^d,+)$.  By the
  proof of Proposition~8.38 in \cite{NIPguide}, we know that $\mu$ is
  related to Haar measure on $\Oo_K^d$ as follows.  Let $D \subseteq
  \Oo^d$ be definable.  Let $X_0, X_{1/2}, X_1 \subseteq
  (\Oo^d,+)/(\Oo^d,+)^{00} \cong \Oo_K^d$ be the following sets:
  \begin{itemize}
  \item $X_0$ is the set of cosets $a + (\Oo^d,+)^{00}$ disjoint from
    $D$.
  \item $X_1$ is the set of cosets $a + (\Oo^d,+)^{00}$ contained in
    $D$.
  \item $X_{1/2}$ is the remaining cosets.
  \end{itemize}
  Compact domination says that the Haar measure of $X_{1/2}$ is zero.
  By the proof of \cite[Proposition~8.38]{NIPguide}, $\mu(D)$ is the
  Haar measure of $X_1$.

  Let $f : \Oo^d \to \Oo^d$ be a definable bijection such that
  \begin{equation*}
    f(\bx) \equiv f(\by) \pmod{p^n\Oo^d} \iff \bx \equiv \by
    \pmod{p^n\Oo^d} \tag{$\ast$}
  \end{equation*}
  for each $n$.  Then $f$ preserves the relation $\bx \equiv \by
  \pmod{(\Oo^d,+)^{00}}$, so $f$ induces a map $\tilde{f}$ on the
  quotient $(\Oo^d,+)/(\Oo^d,+)^{00} \cong \Oo_K^d$.  It's easy to see
  that $\tilde{f}$ is a measure-preserving homeomorphism.  By compact
  domination, $f$ preserves $\mu$.

  For any $a \in \Oo^d$, the left translation $f(x) = a \star x$
  satisfies ($\ast$), because the group $(\Oo^d,\star)$ satisfies
  $\mathfrak{A}_n$.  Therefore $\mu$ is invariant under left
  translations in $(\Oo^d,\star)$.  The measure $\mu$ is also smooth.
  (Smoothness is unrelated to the group structure, and we chose $\mu$
  to be the smooth translation invariant measure on $(\Oo^d,+)$.)

  Therefore $G = (\Oo^d,\star)$ has a smooth measure invariant under
  left translations.  By \cite[Theorem~8.37]{NIPguide}, $G$ is
  compactly dominated.

  We know that $(\Oo^d,+)^{00} = \bigcap_{n=0}^\infty p^n \Oo^d$.  It
  remains to show that $(\Oo^d,\star)^{00} \stackrel{?}{=}
  (\Oo^d,+)^{00}$.  For each $n$, the set $p^n\Oo^d$ is a normal
  subgroup of both groups, and the cosets are the same in both groups
  (Condition $\mathfrak{A}_n$).  Therefore $p^n\Oo^d$ has finite index
  in both groups, and so
  \begin{equation*}
    (\Oo^d,\star)^{00} \subseteq \bigcap_{n
      = 0}^\infty p^n \Oo^d = (\Oo^d,+)^{00}.
  \end{equation*}
  \begin{claim}
    If $D \subseteq \Oo^d$ is a definable set containing $(\Oo^d,\star)^{00}$, then $D
    \div D := \{\bx \star \by^{-1} : \bx, \by \in D\}$ contains
    $(\Oo^d,+)^{00}$.
  \end{claim}
  \begin{claimproof}
    As $D$ contains $(\Oo^d,\star)^{00}$, boundedly many left
    $\star$-translates of $D$ cover $\Oo^d$.  By compactness, finitely
    many left $\star$-translates cover $\Oo^d$.  These translates have
    the same measure with respect to $\mu$, because $\mu$ is
    left-invariant for $\star$.  Then $D$ must have positive
    $\mu$-measure.  Because of the connection between $\mu$ and Haar
    measure, some additive coset $\ba + (\Oo^d,+)^{00}$ must lie in
    $D$.  By compactness, there is some $n$ such that $\ba + p^n\Oo^d
    \subseteq D$.  By Condition $\mathfrak{A}_n$, the set $\ba +
    p^n\Oo^d$ is a $\star$-coset of $p^n \Oo^d \lhd (\Oo^d,\star)$.
    Then $D \div D$ contains $p^n \Oo^d \supseteq (\Oo^d,+)^{00}$.
  \end{claimproof}
  Because $(\Oo^d,\star)^{00}$ is a type-definable subgroup of
  $(\Oo^d,\star)$, we have
  \begin{gather*}
    (\Oo^d,\star)^{00} = (\Oo^d,\star)^{00} \div (\Oo^d,\star)^{00} \\
    =
    \bigcap \{D \div D : D \text{ is definable and } D \supseteq
    (\Oo^d,\star)^{00}\} \stackrel{\text{(Claim)}}{\supseteq} (\Oo^d,+)^{00}.
  \end{gather*}
  The second equality holds by compactness: if $\Sigma(\bx)$ is the
  partial type defining $(\Oo^d,\star)^{00}$ and $\bc$ is in $\Oo^d$,
  then the following are equivalent:
  \begin{itemize}
  \item $\bc \in (\Oo^d,\star)^{00}$.
  \item There are $\ba, \bb \in (\Oo^d,\star)^{00}$ such that $\bc =
    \ba \star \bb^{-1}$.
  \item The partial type in the variables $\bx, \by$ saying
    \begin{equation*}
      \Sigma(\bx) \text{ and } \Sigma(\by) \text{ and } \bc = \bx \star \by^{-1}
    \end{equation*}
    is finitely satisfiable.
  \item $\bc \in \bigcap \{\phi(\Mm) \div \phi(\Mm) : \phi \in \Sigma\}$.
  \item $\bc \in \bigcap \{D \div D : D \text{ is definable and } D \supseteq
    (\Oo^d,\star)^{00}\}$.
  \end{itemize}
  So we conclude that
  \begin{equation*}
    (\Oo^d,\star)^{00} = (\Oo^d,+)^{00} = \bigcap_{n = 0}^\infty
    p^n\Oo^d. \qedhere
  \end{equation*}
\end{proof}

\subsection{Getting a $p$-adic Lie group}
Continue to work in a monster model $\Mm$ of a $P$-minimal expansion $T$
of $\Th(K)$ for some finite extension $K/\Qq_p$.
\begin{proposition} \label{b-good}
  Suppose $G = (\Oo^d,\star)$ satisfies condition
  $\mathfrak{B}_\omega$.  Then $G$ is fsg and compactly dominated, and
  the compact topological group $G/G^{00}$ is isomorphic to $\Oo_K^d$.
\end{proposition}
\begin{proof}
  By Proposition~\ref{a-b}, $G^{00}$ is the subgroup $\bigcap_{n = 0}^\infty
  p^n \Oo^d$, and so $G/G^{00}$ is the inverse limit
  \begin{equation*}
    \varprojlim_{n \to \infty} (\Oo^d,\star)/p^n\Oo^d.
  \end{equation*}
  By condition $\mathfrak{B}_\omega$, this is the same as
  \begin{equation*}
    \varprojlim_{n \to \infty} (\Oo^d,+)/p^n\Oo^d,
  \end{equation*}
  which is just $\Oo_K^d$.
\end{proof}
\begin{proposition} \label{c-good}
  Suppose $G = (\Oo^d,\star)$ satisfies condition
  $\mathfrak{C}_\omega$.  Then $G$ is fsg and compactly dominated, and $G/G^{00}$ is isomorphic to a $d$-dimensional Lie group over $K$.
\end{proposition}
\begin{proof}
  Proposition~\ref{a-b} gives the first part, and shows that $G^{00} =
  \bigcap_{k=0}^\infty p^k\Oo^n$.  Then $G/G^{00}$ is homeomorphic to
  $\Oo_K^n$ via the standard part map $\st : \Oo \to \Oo_K$.  That is,
  the following two maps are equivalent:
  \begin{gather*}
    G \to G/G^{00} \\
    \Oo^n \stackrel{st}{\to} \Oo_K^n.
  \end{gather*}
  Condition $\mathfrak{C}_\omega$ shows that the induced group
  structure on $\Oo_K^n$ is given by
  \begin{gather*}
    \star : \Oo_K^n \times \Oo_K^n \to \Oo_K^n \\
    \bx \star \by = \sum_{I,J} \st(\bc_{I,J}) \bx^I \by^J.
  \end{gather*}
  Therefore the induced structure is a Lie group of dimension $n$.
\end{proof}
\begin{proposition} \label{ed-good}
  Suppose $G = (\Oo^d,\star)$ satisfies the conditions
  $\mathfrak{A}_\omega$, $\mathfrak{B}_1$, $\mathfrak{B}_2$,
  $\mathfrak{D}$, and $\mathfrak{E}$.  Then $G$ is fsg and compactly
  dominated, and $G/G^{00}$ is isomorphic to a $de$-dimensional Lie
  group over $\Qq_p$, where $e = [K : \Qq_p]$.
\end{proposition}
\begin{proof}
  Proposition~\ref{a-b} gives the first part, and shows that $G^{00} =
  \bigcap_{k=0}^\infty p^k\Oo^n$.  Then $G/G^{00}$ is canonically
  isomorphic to a topological group $(\Oo_K^d,\ast)$.  The assumptions
  on $\star$ precisely imply that $\ast$ satisfies all the
  requirements in Lemma~\ref{recognizer2}, and so $G/G^{00}$ is a
  $de$-dimensional Lie group over $\Qq_p$.
\end{proof}

\section{The first two proofs of the Onshuus-Pillay conjecture} \label{again}
In this section, we prove Theorems~\ref{second-main} and
\ref{third-main}, the first two versions of the Onshuus-Pillay
conjecture.

\begin{lemma}\label{stick}
  Suppose $M \models T$ and $D \subseteq M^n$ is definable, with
  $\dim(D) = k$.  Then there is a definable injection $D \to M^k$.
\end{lemma}
\begin{proof}
  Write $X \le Y$ if there is a definable injection from $X$ to $Y$.
  Note that $M \le \Oo \times \{0,1\}$, because of the map
  \begin{gather*}
    f(x) = 
    \begin{cases}
      (x,0) & \text{ if } x \in \Oo \\
      (1/x,1) & \text{ if } x \notin \Oo.
    \end{cases}
  \end{gather*}
  Thus $M \le \Oo + \Oo$, where $+$ denotes disjoint union.  On the
  other hand, $\Oo \le B$ and $B \le \Oo$ for any ball $B$, and we can
  find two disjoint balls inside $\Oo$, so $\Oo + \Oo \le B + B \le
  \Oo$.  Then $M \le \Oo$, so
  \begin{equation*}
    M + M \le \Oo + \Oo \le \Oo \le M.
  \end{equation*}
  Multiplying by $M^{n-1}$, we see that $M^k + M^k \le M^k$ for any
  $k$.  And of course $M^j \le M^k$ for $j < k$.  In light of this, it
  suffices to cover $D$ with finitely many definable sets $D_i$, such
  that $D_i \le M^{k_i}$ for some $k_i \le \dim(D)$.  Such a
  decomposition is provided by the ``topological cell decomposition''
  of \cite{p-minimal-cells}.
\end{proof}
If $(G,\star)$ is a definable group and $a \in G$, define
\begin{gather*}
  \star_a : G \times G \to G \\
  x \star_a y = x \star a^{-1} \star y.
\end{gather*}
Then $(G,\star_a)$ is a definable group with identity element $a$,
definably isomorphic to $(G,\star)$ by the following
isomorphism:
\begin{gather*}
  (G,\star) \to (G,\star_a) \\
  x \mapsto a \star x.
\end{gather*}
In particular, if we prove Conjecture~\ref{op-conj} for $(G,\star_a)$,
then we prove it for $(G,\star)$.
\begin{remark}
  In the proof of the following lemma, we use the notion $\dim(\ba/B)$
  where $\ba$ is a finite tuple and $B$ is a small set of parameters.
  Recall that in any theory $T$ where $\acl$ satisfies the exchange property,
  one defines the \emph{dimension} $\dim(a_1,\ldots,a_n/B)$ as the maximum cardinality of a
  subset $\{c_1,\ldots,c_m\} \subseteq \{a_1,\ldots,a_n\}$ which is
  \emph{independent} over $B$, in the sense that $c_i$ isn't algebraic over $B \cup \{c_1,\ldots,c_{i-1}\}$ for $i=1,\ldots,n$ (the order of the $c_i$ doesn't matter).  Moreover, the \emph{dimension} $\dim(D)$ of a definable
  set $D \subseteq \Mm^n$ is defined to be $\max_{\ba \in D}
  \dim(\ba/B)$ for any small set $B \subseteq \Mm$ defining $D$ (the
  choice of $B$ doesn't matter).  Conversely, $\dim(\ba/B)$ is the
  minimum of $\dim(D)$ as $D$ ranges over $B$-definable sets
  containing $\ba$.  See \cite[Section~2]{HP94} for references for these facts.  In $P$-minimal theories, $\acl$ satisfies the exchange property \cite[Theorem~6.2]{p-min}, and the acl-dimension of a definable set $D$ agrees with the usual dimension $\dim(D)$ \cite[Theorem~6.3]{p-min}.

  To summarize:
  \begin{itemize}
  \item $\dim(\ba/B)$ is the minimum of $\dim(D)$ as $D$ ranges over
    $B$-definable sets $D \ni \ba$.
  \item If $B$ is small and $D \subseteq \Mm^n$ is $B$-definable, then
    \[\dim(D) = \max_{\ba \in D} \dim(\ba/D).\]
  \end{itemize}
  We will also use the well-known facts
  \begin{gather*}
    \dim(\ba,\bb/C) = \dim(\ba/C\bb) + \dim(\bb/C) \\
    \acl(C\ba) = \acl(C\bb) \implies \dim(\ba/C) = \dim(\bb/C).
  \end{gather*}
\end{remark}
\begin{lemma} \label{local-C1}
  Let $G \subseteq \Mm^n$ be an $n$-dimensional definable set and
  $\star$ be a definable group operation on $G$.  Then there is a
  definable set $U$ in the interior of $G$ with $\dim (G \setminus U)
  < n$, such that for any $a \in U$, the group operation $\star_a$ is
  strictly differentiable on a neighborhood of $(a,a)$.
\end{lemma}
\begin{proof}
  Morally, this comes from the fact that definable functions in $\Mm$
  are generically strictly differentiable (Theorem~\ref{generic-diff}).
  However, the proof is not completely trivial, and we should give
  details.

  Before giving the real proof, we first sketch the conceptually
  ``correct'' proof.  First, construct a canonical strict
  $C^{1}$-manifold structure on $G$ compatible with the group
  structure.  Then apply generic differentiability to the identity map
  from $G$ (with its strict $C^{1}$-manifold structure) to $G$ (as a
  subset of $\Mm^n$).  Fix a point $a$ where this map is strictly
  differentiable.  Then an open neighborhood $a \in U \subseteq \Mm^n$
  is also a local chart for the manifold structure.  (That is, the two
  strict $C^1$-structures on $G$ agree with each other close to the
  point $a$.)  Because the group structure is strict $C^1$, the
  function $(x,y) \mapsto x \star a^{-1} \star y$ is strictly
  differentiable.
  
  Rather than formally developing strict $C^1$-manifolds, we instead
  give a more direct proof.  By generic strict differentiability,
  there is a definable open set $\Delta$ in the interior of $G \times
  G \subseteq \Mm^{2n}$, such that $\dim((G \times G) \setminus
  \Delta) < \dim(G \times G) = 2n$, and the group operations are
  strictly differentiable on $\Delta$.  Let $\nabla$ be the set of
  pairs $(a,b) \in G \times G$ such that the following conditions
  hold:
  \begin{gather*}
    (b^{-1},a) \in \Delta \\
    (b^{-1} \star a, a^{-1}) \in \Delta \\
    (b,b^{-1} \star a) \in \Delta.
  \end{gather*}
  Let $M_0$ be a small model over which everything is defined.  If the
  pair $(a,b) \in G^2$ is jointly generic, in the sense that
  $\dim(a,b/M_0) = 2\dim(G)$, then so are the pairs $(b^{-1},a),
  (b^{-1} \star a, a^{-1}), (b,b^{-1} \star a)$, and so all three
  pairs belong to $\Delta$, and $(a,b) \in \nabla$.  Let $U$ be the
  projection of $\nabla$ onto the first coordinate, intersected with
  the interior of $G$:
  \begin{equation*}
    U = \{a : (a,b) \in \nabla\} \cap \ter(G).
  \end{equation*}
  If $a \in G$ is generic over $M_0$, we can find $b \in G$ generic
  over $M_0a$, and then $(a,b)$ is generic over $M_0$, $(a,b) \in
  \nabla$, and $a \in U$.  It follows that $U$ contains all $a \in G$
  which are generic over $M_0$, and so $\dim(G \setminus U) <
  \dim(G)$.

  For $a \in U$, we will prove that $\star_a$ is strictly
  differentiable around $a$.  Take $b \in G$ such that $(a,b) \in
  \nabla$.  Then the composition
  \begin{equation*}
    b \star \left(\left((b^{-1} \star x) \star a^{-1}\right) \star
    y\right) = x \star_a y
  \end{equation*}
  is strictly differentiable for $(x,y)$ sufficiently close to $(a,a)$, because for
  such $x$ and $y$ we have
  \begin{gather*}
    (b^{-1},x) \approx (b^{-1},a) \in \Delta \\
    \left((b^{-1} \star x),a^{-1}\right) \approx (b^{-1} \star a, a^{-1}) \in \Delta \\
    \left(\left((b^{-1} \star x) \star a^{-1}\right),
    y\right) \approx (b^{-1} \star a \star a^{-1},a) = (b^{-1},a) \in \Delta \\
    \left(b, \left(\left((b^{-1} \star x) \star a^{-1}\right) \star
    y\right)\right) \approx (b,b^{-1} \star a) \in \Delta,
  \end{gather*}
  where $\approx$ means ``arbitrarily close to''.
\end{proof}
% TODO: I jumped through hoops to ensure that not only does one $a$ work, but actually almost all $a$ work.  But... I don't think I even needed that!  Unless it's used in the third proof.
\begin{lemma} \label{to-b-2}
  Let $G \subseteq \Mm^n$ be an $n$-dimensional definable set and
  $\star$ be a definable group operation on $G$.  Let $a$ be the
  identity element of $G$, and suppose $a$ is in the interior of $G$
  and $\star$ is strictly differentiable at $(a,a)$.  Then there is
  some non-zero $\epsilon$ such that $a + \epsilon \Oo^n$ is a
  subgroup of $G$ and $a + \epsilon \Oo^n$ satisfies conditions
  $\mathfrak{A}_\infty$, $\mathfrak{B}_\omega$, $\mathfrak{D}$ and
  $\mathfrak{E}$.  In fact, any sufficiently small $\epsilon$ works.
\end{lemma}
\begin{proof}
  Moving $G$ by a translation, we may assume $a = \bar{0}$.  Let $f :
  G \to G$ be the $p$th power map.  A straightforward calculation
  shows that at $\bar{0}$, the functions $\bx \star \by$, $\bx^{-1}$,
  and $f(\bx)$ have the following strict derivatives:
  \begin{gather*}
    \frac{\partial}{\partial \bx} (\bx \star \by) = I_n \\
    \frac{\partial}{\partial \by} (\bx \star \by) = I_n \\
    \frac{\partial}{\partial \bx} \bx^{-1} = -I_n \\
    \frac{\partial}{\partial \bx} f(\bx) = pI_n,
  \end{gather*}
  where $I_n$ is the $n \times n$ identity matrix.\footnote{For the
  group inverse map $\bx^{-1}$, strict differentiability follows by
  applying the inverse function theorem, Theorem~\ref{ift}, to the map
  $(\bx,\by) \mapsto (\bx,\bx \star \by)$.}

  Take a small model $M$ defining $(G,\star)$.  Let $\epsilon \in \Mm$
  be a non-zero $M$-infinitesimal, i.e., a non-zero element satisfying
  the equivalent conditions:
  \begin{itemize}
  \item $\epsilon$ is contained in any $M$-definable neighborhood of
    0.
  \item $\epsilon$ is contained in any $M$-definable ball
    $B_\gamma(0)$.
  \item $v(\epsilon) > \Gamma_M$.
  \end{itemize}
  Note that $\epsilon a$ is an $M$-infinitesimal for any $a \in \Oo =
  \Oo_\Mm$, including $a$ outside of $M$.

  We first verify that $\epsilon \Oo^n$ is a subgroup of $G$.  Let
  $\bx, \by$ be tuples in $\epsilon \Oo^n$.  Then $\bx$ and $\by$ are
  tuples of $M$-infinitesimals.  By differentiability,
  \begin{equation*}
    \bx \star \by = \bx + \by + \mu_1 \bx + \mu_2 \by
  \end{equation*}
  for some $n \times n$ matrices $\mu_1, \mu_2$ of $M$-infinitesimals.
  The right hand side is in $\epsilon \Oo^n$.  A similar argument
  shows that $\bx \in \epsilon \Oo^n \implies \bx^{-1} \in \epsilon
  \Oo^n$.  Now we check each of the required conditions:
  \begin{itemize}
  \item[$(\mathfrak{A}_\infty)$] Let $I$ be a principal ideal of $\Oo$
    contained in $\epsilon \Oo$.  We must show that the group
    operation $\star$ on $\epsilon \Oo^n$ respects congruence modulo
    $I^n$.  Suppose $\bx, \by, \bz \in \epsilon \Oo^n$ with $\bx
    \equiv \by \pmod{I^n}$.  Then $\bx, \by, \bz$ are tuples of
    $M$-infinitesimals.  By \emph{strict} differentiability of the group
    operation,
    \begin{equation*}
      (\bx \star \bz) - (\by \star \bz) = (\bx - \by) + \mu (\bx - \by)
    \end{equation*}
    for some matrix $\mu$ of $M$-infinitesimals.  The right hand side
    is in $I^n$, so $\bx \star \bz \equiv \by \star \bz \pmod{I^n}$.
    Similar arguments show
    \begin{gather*}
      \by \equiv \bz \pmod{I^n} \implies \bx \star \by \equiv \bx \star \bz \pmod{I^n} \\
      \bx \equiv \by \pmod{I^n} \implies \bx^{-1} \equiv \by^{-1} \pmod{I^n}.
    \end{gather*}
    Thus, congruence modulo $I^n$ is a congruence in the group
    $(\epsilon \Oo^n, \star)$, which is condition
    $(\mathfrak{A}_\infty)$.
  \item[$(\mathfrak{B}_\omega)$] If $\bx, \by$ are tuples in $\epsilon
    \Oo^n$, then $\bx, \by$ are $M$-infinitesimals, so
    differentiability of $\star$ gives
    \begin{equation*}
      \bx \star \by = \bx + \by + \mu_1 \bx + \mu_2 \by
    \end{equation*}
    for some $M$-infinitesimal matrices $\mu_1, \mu_2$.  The entries
    in $\mu_1, \mu_2$ have valuation greater than $v(p^k)$ for any
    $k$, so
    \begin{equation*}
      \bx \star \by - \bx - \by \in p^k \epsilon \Oo^n.
    \end{equation*}
    This is condition $(\mathfrak{B}_k)$, for arbitrary $k$.
  \item[$(\mathfrak{D})$] If $\bx, \by$ are in $\epsilon \Oo^n$, then
    they are $M$-infinitesimal, and so the strong differentiability of
    $f$ shows that
    \begin{equation*}
      f(\bx) - f(\by) = p(\bx - \by) + \mu(\bx - \by)
    \end{equation*}
    for some matrix $\mu$ of $M$-infinitesimals.  The right hand side
    has valuation $v(p) + v(\bx - \by)$ because $\mu$ is infinitesimal
    compared to $p$.  Thus condition $\mathfrak{D}$ holds.
  \item[$(\mathfrak{E})$] The derivative of $f$ at $\bar{0}$ is
    $pI_n$, which is invertible, so the inverse function theorem
    (Theorem~\ref{ift}) shows that $f$ maps the set of
    $M$-infinitesimal vectors onto the set of $M$-infinitesimal
    vectors.  If $\bx \in p \epsilon \Oo^d$, then $\bx$ is an
    $M$-infinitesimal vector, so $\bx = f(\by)$ for some
    $M$-infinitesimal vector $\by$.  By the differentiability of $f$ at
    $\bar{0}$,
    \begin{equation*}
      \bx = f(\by) = p \by + \mu \by
    \end{equation*}
    for some $M$-infinitesimal matrix $\mu$.  Then $v(\bx) = v(p) +
    v(\by)$ so $\by \in p^{-1}p \epsilon \Oo^n = \epsilon \Oo^n$.
    Thus every element of $p \epsilon \Oo^n$ is a $p$th power of an
    element in $\epsilon \Oo^n$, which is condition $\mathfrak{E}$. \qedhere
  \end{itemize}
\end{proof}
\begin{theorem} \label{second-main}
  Let $G$ be an $n$-dimensional definable group in a highly saturated
  model $\Mm$ of a $P$-minimal expansion $T$ of $\Th(K)$ for some finite
  extension $K/\Qq_p$.  Then there is a definable open subgroup $H
  \subseteq G$ such that $H$ is fsg, $H/H^{00} \cong \Oo_K^n$, and $H$
  is compactly dominated by $H/H^{00}$.  In particular, $H/H^{00}$ is
  an $n$-dimensional Lie group over $K$.
\end{theorem}
\begin{proof}
  By Lemma~\ref{stick}, we may assume $G \subseteq \Mm^n$.  By
  Lemma~\ref{local-C1} we may assume that the group operation is
  differentiable at the identity element.  Lemma~\ref{to-b-2} gives an
  $n$-dimensional definable subgroup $H \subseteq G$ which (up to
  isomorphism) satisfies $\mathfrak{A}_\infty$, $\mathfrak{B}_\omega$,
  $\mathfrak{D}$, and $\mathfrak{E}$.  In particular, it satisfies
  $\mathfrak{B}_\omega$.  Proposition~\ref{b-good} shows that $H$ has the
  desired properties.
\end{proof}
\begin{theorem} \label{third-main}
  Let $M$ be a model of a $P$-minimal expansion $T$ of $\Th(K)$ for
  some finite extension $K/\Qq_p$ with $e = [K : \Qq_p]$.  In a
  monster model $\Mm \succeq M$, let $G$ be an $n$-dimensional
  $M$-definable group.  Then there is an $M$-definable open subgroup
  $H \subseteq G$ such that $H$ is fsg, $H/H^{00}$ is an
  $ne$-dimensional Lie group over $\Qq_p$ and $H$ is compactly
  dominated by $H/H^{00}$.
\end{theorem}
\begin{proof}
  As in Theorem~\ref{second-main}, we can find a $n$-dimensional
  definable subgroup $H \subseteq G$ which satisfies
  $\mathfrak{A}_\infty$, $\mathfrak{B}_\omega$, $\mathfrak{D}$, and
  $\mathfrak{E}$.  In particular, it satisfies the combination
  $\mathfrak{A}_\infty \wedge \mathfrak{B}_1 \wedge \mathfrak{B}_2 \wedge \mathfrak{D}
  \wedge \mathfrak{E}$.  This combination of properties is definable
  (Remark~\ref{def-el}), so we can change $H$ to be $M$-definable.
  Then Proposition~\ref{ed-good} shows that $H$ has the desired properties.
\end{proof}
From the proofs, one can extract the following observation, which is
similar to a theorem of Acosta L\'opez on definable groups in
1-h-minimal theories \cite[Proposition~6.3]{acosta-defcom}.
\begin{observation}
  Let $G$ be an $n$-dimensional definable group in a model $M$ of a
  $P$-minimal theory.  Then there is a definable family
  $\{U_\gamma\}_{\gamma \in \Gamma}$ such that each $U_\gamma$ is an
  $n$-dimensional definable subgroup,
    $\gamma \le \gamma' \implies U_\gamma \supseteq U_{\gamma'}$,
  and $\bigcap_{\gamma \in \Gamma} U_\gamma = \{1\}$.
\end{observation}
\begin{proof}
  Replacing $M$ with an elementary extension, we may assume that $M$
  is highly saturated.  (This uses the fact that dimension is
  definable in families.)  By Lemma~\ref{stick}, we may assume $G
  \subseteq \Mm^n$.  By Lemma~\ref{local-C1} we may assume that the
  group operation is differentiable at the identity element.
  Lemma~\ref{to-b-2} gives an $n$-dimensional definable subgroup $H
  \subseteq G$ which (up to isomorphism) satisfies
  $\mathfrak{A}_\infty$, among other things.  Condition
  $\mathfrak{A}_\infty$ implies that the family of $n$-dimensional
  balls around 0 has all the desired properties.
\end{proof}

\section{Third proof, for pure $p$-adically closed fields} \label{thpr}
In this section, we restrict to the case where $T = \Th(K)$, i.e.,
pure $p$-adically closed fields.  Recall that $K$ is a fixed $p$-adic
field, i.e., a finite extension of $\Qq_p$.  We assume the language
$\Ll$ contains constant symbols naming a basis of $K$ over $\Qq_p$.
This ensures that the theory has definable Skolem functions and that
$\dcl(\varnothing)$ is dense in $K$.  The density implies that every
$K$-definable open ball is 0-definable, a fact we will need later.

Our goal is to prove Theorem~\ref{last-main}, the third version of the
Onshuus-Pillay conjecture.  By Proposition~\ref{c-good}, we
essentially reduce to showing that definable groups locally satisfy
condition $\mathfrak{C}_\omega$ from Definition~\ref{ABC}.  When
everything is defined over the standard model and the group law is an
analytic function, Example~\ref{c-source} makes things work
correctly.  The main difficulty will be transferring this argument
from the standard model to the monster model.  In principle, this
requires keeping careful track of the analyticity of definable
functions in $K$ and their radii of convergence.

\subsection{Splendid and locally splendid functions}
In this subsection, we introduce a class of definable functions $\Oo^n
\to \Oo^m$ called \emph{splendid functions}
(Definition~\ref{splendid-def}).  Morally, a function $f : \Oo^n \to
\Oo^m$ is splendid if $f$ is given by a power series with coefficients
in $\Oo$, in some non-standard sense.  We will show in the next
subsection that definable functions are locally splendid at generic
points (Proposition~\ref{gen-lsa}).

The machinery of splendid and locally splendid functions can be seen
as a model-theoretic hack to avoid keeping track of the radii of
convergence of definable functions.

\begin{definition}
  A function $f : \Oo_K^n \to \Oo_K^m$ is \emph{pre-splendid} if it's
  given by a convergent power series with coefficients in $\Oo_K$.
  That is,
  \begin{equation*}
    f(\bx) = \sum_I \bc_I \bx^I
  \end{equation*}
  where $\bc_I \in \Oo^m$ and $\lim_{|I| \to \infty} v(\bc_I) =
  +\infty$.
\end{definition}
\begin{observation} \phantomsection \label{ps-obs}
  \begin{enumerate}
  \item \label{po1} If $f : \Oo_K^n \to \Oo_K^m$ is constant, then $f$ is
    pre-splendid.
  \item \label{po5} The identity function $\Oo_K^n \to \Oo_K^n$ and coordinate
    projections $\pi_1,\ldots,\pi_n : \Oo_K^n \to \Oo_K$ are
    pre-splendid.  More generally, any polynomial map $f : \Oo_K^n \to
    \Oo_K^m$ is pre-splendid, \emph{provided} that the coefficients
    are in $\Oo_K$.    
  \item \label{po2} If $f : \Oo_K^n \to \Oo_K^m$ and $g : \Oo_K^m \to \Oo_K^\ell$
    are pre-splendid, then the composition $g \circ f : \Oo_K^n \to
    \Oo_K^\ell$ is pre-splendid.
%% I wrote it down and it does work.  It's easiest to think in terms of a
%% lemma saying that if I is some infinite set, and a_i are elements of O
%% which tend to 0 as i goes to infty in the one-point compactification
%% of I, and if f_i is a pre-splendid function in some variables xbar,
%% then there's a presplendid function g(xbar) = sum_i a_i*f_i(xbar).
%% You can use this to easily check that products and compositions of
%% presplendid functions are presplendid.
  \item \label{po3} If $f : \Oo_K^n \to \Oo_K^m$ is a function, then $f$ is
    pre-splendid if and only if the component functions $f_1, \ldots,
    f_m : \Oo_K^n \to \Oo_K$ are pre-splendid.
  \item \label{po4} The set of pre-splendid functions from $\Oo_K^n$
    to $\Oo_K$ is an $\Oo_K$-algebra, i.e., it is closed under the
    ring operations and multiplication by $\Oo_K$.
  \item \label{po6} If $f : \Oo_K^n \to \Oo_K^m$ is pre-splendid, then $f$ is
    strictly differentiable, and the strict derivative $Df : \Oo_K^n
    \to \Oo_K^{nm}$ is pre-splendid.
  \end{enumerate}
\end{observation}
Now consider a monster model $\Mm \succ K$.  Recall that $\Oo$ denotes
$\Oo_\Mm$.
\begin{definition} \label{splendid-def}
  Let $D$ be a 0-definable set and $\mathcal{F} = \{f_a\}_{a \in D}$ be a 0-definable family of
  functions from $\Oo^n \to \Oo^m$.  Then $\mathcal{F}$ is a
  \emph{splendid family} if for every $a \in D(K)$, the function
  $f_a(K) : \Oo_K^n \to \Oo_K^m$ is pre-splendid.  A definable
  function $f : \Oo^n \to \Oo^m$ is \emph{splendid} if it belongs to
  some splendid family.
\end{definition}
We can transfer the properties of pre-splendid functions in
Observation~\ref{ps-obs} to splendid functions:
\begin{proposition} \phantomsection \label{splendid-props}
  \begin{enumerate}
  \item If $f : \Oo^n \to \Oo^m$ is constant, then $f$ is splendid.
  \item If $f : \Oo^n \to \Oo^m$ is a polynomial map whose
    coefficients are in $\Oo$, then $f$ is splendid.  For example, the
    identity map $\Oo^n \to \Oo^n$ and coordinate projections
    $\pi_1,\ldots,\pi_n : \Oo^n \to \Oo$ are splendid.    
  \item If $f : \Oo^n \to \Oo^m$ and $g : \Oo^m \to \Oo^\ell$ are
    splendid, then the composition $g \circ f : \Oo^n \to \Oo^\ell$ is
    splendid.
  \item If $f : \Oo^n \to \Oo^m$ is a definable function, then $f$ is
    splendid if and only if the component functions $f_1,\ldots,f_m :
    \Oo^n \to \Oo$ are splendid.
  \item The set of splendid functions from $\Oo^n$ to $\Oo$ is an
    $\Oo$-algebra, i.e., it is closed under the ring operations and
    multiplication by $\Oo$.
  \item If $f : \Oo^n \to \Oo^m$ is splendid, then $f$ is strictly
    differentiable, and the strict derivative $Df : \Oo^n \to
    \Oo^{nm}$ is splendid.
  \end{enumerate}
\end{proposition}
\begin{proof}
  \begin{enumerate}
  \item The family of constant functions from $\Oo^n \to \Oo^m$ is a
    splendid family, by Observation~\ref{ps-obs}(\ref{po1}).
  \item Suppose $f : \Oo^n \to \Oo^m$ is polynomial, with all
    coefficients in $\Oo$.  Let $d$ be the degree of $f$, and let
    $\mathcal{F}_{\le d}$ be the family of all polynomial maps $g :
    \Oo^n \to \Oo^m$ such that the coefficients of $g$ are in $\Oo$
    and $\deg(g) \le d$.  Then $\mathcal{F}_{\le d}$ is a 0-definable
    family (because of the bound on degree) and $\mathcal{F}_{\le d}$
    is a splendid family by Observation~\ref{ps-obs}(\ref{po5}).
  \item Take splendid families $\mathcal{F} \ni f$ and $\mathcal{G}
    \ni g$.  By Observation~\ref{ps-obs}(\ref{po2}), the family $\{g'
    \circ f' : f' \in \mathcal{F}, ~ g' \in \mathcal{G}\}$ is a
    splendid family containing $g \circ f$.
  \item Similar.
  \item Similar.
  \item
    Take a splendid family $\mathcal{F} \ni f$.  By
    Observation~\ref{ps-obs}(\ref{po6}), every function in
    $\mathcal{F}(K)$ is strictly differentiable.  Because $K \preceq
    \Mm$, this implies the strict differentiability of functions in
    $\mathcal{F}$ (including $f$).  Let $\mathcal{F}' = \{Dg : g \in
    \mathcal{F}\}$, the set of strict derivatives of functions in
    $\mathcal{F}$.  Then $\mathcal{F}'$ is a splendid family by
    Observation~\ref{ps-obs}(\ref{po6}), so $Df$ is splendid. \qedhere
  \end{enumerate}
\end{proof}
Additionally, 0-definable pre-splendid functions are splendid:
\begin{proposition} \label{lazy}
  If $f : \Oo^n \to \Oo^m$ is 0-definable and $f(K) : \Oo_K^n \to
  \Oo_K^m$ is pre-splendid, then $f$ is splendid.
\end{proposition}
\begin{proof}
  The singleton family $\{f\}$ is a splendid family.
\end{proof}
\begin{lemma} \label{scaler}
  Let $f : \Oo^n \to \Oo^m$ be splendid, with $f(\bar{0}) = \bar{0}$.
  Suppose $\epsilon \in \Oo$ is non-zero.  Let $g : \Oo^n \to \Mm^m$
  be the function $g(\bx) = f(\epsilon \bx)/\epsilon$.  Then $\img(g)
  \subseteq \Oo^m$, and the function $g : \Oo^n \to \Oo^m$ is
  splendid.
\end{lemma}
\begin{proof}
  It suffices to check the analogous statement for pre-splendid functions.
  Let $f : \Oo_K^n \to \Oo_K^m$ be pre-splendid with $f(\bar{0}) =
  \bar{0}$.  Then $f(\bx) = \sum_{I} \bc_I \bx^I$, where each $\bc_I$
  is in $\Oo_K^m$ and the constant term $\bc_{\bar{0}}$ vanishes.
  Then
  \begin{equation*}
    g(\bx) = f(\epsilon \bx)/\epsilon = \sum_I \epsilon^{|I|-1} \bc_I
    \bx^I.
  \end{equation*}
  All the coefficients $\epsilon^{|I|-1} \bc_I$ are still in $\Oo$,
  since the constant term vanishes.  Thus $g$ is pre-splendid.
%%   Now we verify the original statement for splendid functions.  Take a
%%   splendid family $\mathcal{F} \ni f$.  Let $\mathcal{G}$ be the
%%   family of functions $\Oo^n \to \Mm^m$ of the form $f'(\epsilon
%%   \bx)/\epsilon$ for $f' \in \mathcal{F}$ and $\epsilon \in \Oo
%%   \setminus \{0\}$.  We have just shown that every $K$-definable
%%   function in $\mathcal{G}$ is pre-splendid.  It follows that
%%   $\mathcal{G}$ is a splendid family (and in particular that $\img(g)
%%   \subseteq \Oo^m$ for every $g \in \mathcal{G}$).  Then $\mathcal{G}$
%%   witnesses that the given function $g(\bx) = f(\epsilon
%%   \bx)/\epsilon$ is splendid.
\end{proof}
\begin{definition} \label{ls-def}
  Let $U \subseteq \Mm^n$ be open and $f : U \to \Mm^m$ be a definable function.    Then $f$ is \emph{locally splendid}
  at $a \in U$ if there are some non-zero $\epsilon$ and $\delta$ such
  that
  \begin{gather*}
    a + \epsilon \Oo^n \subseteq U \\
    f(a + \epsilon \Oo^n) \subseteq f(a) + \delta \Oo^m
  \end{gather*}
  and the function
  \begin{equation*}
    a + \epsilon \Oo^n \stackrel{f}{\to} f(a) + \delta \Oo^m
  \end{equation*}
  is splendid, or more precisely, the following function is
  splendid:
  \begin{gather*}
    \Oo^n \to \Oo^m \\
    x \mapsto (f(a + \epsilon x) - f(a)) \delta^{-1}.
  \end{gather*}
  Finally, we say that $f$ is \emph{locally splendid} if it is
  splendid at every $a \in U$.
\end{definition}
Note that the definition of ``locally splendid'' smuggles in some
uniformity conditions, because of the saturation of $\Mm$.  For
example, if $f$ is locally splendid, there must be a single splendid
family which witnesses splendidness at every $a \in U$.  We will not
use this fact, however.

We first carry out some sanity checks on Definition~\ref{ls-def}:
splendid functions are locally splendid (Proposition~\ref{sanity}),
and local splendidness is a local property
(Proposition~\ref{ls-local}).
\begin{proposition} \label{sanity}
  If $f : \Oo^n \to \Oo^m$ is splendid, then $f$ is locally splendid.
\end{proposition}
\begin{proof}
  For any $a \in \Oo^n$, take $\epsilon = \delta = 1$.  Then
  \begin{gather*}
    a + \Oo^n = \Oo^n = \dom(f) \\
    f(a + \Oo^n) = f(\Oo^n) = \img(f) \subseteq \Oo^m = f(a) + \Oo^m,
  \end{gather*}
  and the function
  \begin{gather*}
    \Oo^n \to \Oo^m \\
    x \mapsto f(a + x) - f(a)
  \end{gather*}
  is splendid by Proposition~\ref{splendid-props}.
\end{proof}
\begin{lemma} \label{sanity2}
  Fix a definable function $f : U \to \Mm^n$ and a point $a \in U$.
  Say that a pair $(\epsilon,\delta)$ is ``suitable'' if it satisfies
  the conditions in Definition~\ref{ls-def}:
  \begin{gather*}
    a + \epsilon \Oo^n \subseteq U \\
    f(a + \epsilon \Oo^n) \subseteq f(a) + \delta \Oo^m
  \end{gather*}
  and the following function is splendid:
  \begin{gather*}
    \Oo^n \to \Oo^m \\
    x \mapsto (f(a + \epsilon x) - f(a)) \delta^{-1}.
  \end{gather*}
  Suppose that $f$ is locally splendid at $a$, i.e., some pair
  $(\epsilon_0,\delta_0)$ is suitable.
  \begin{enumerate}
  \item If $(\epsilon,\delta)$ is suitable, then $(\epsilon',\delta)$
    is suitable for any $\epsilon'$ with $v(\epsilon') \ge
    v(\epsilon)$.
  \item If $(\epsilon,\delta)$ is suitable, then $(\epsilon,\delta')$
    is suitable for any $\delta'$ with $v(\delta') \le v(\delta)$.
  \item For any $\delta$, there is some $\epsilon$ such that
    $(\epsilon,\delta)$ is suitable.
  \end{enumerate}
\end{lemma}
\begin{proof}
  \begin{enumerate}
  \item Let $g : \Oo^n \to \Oo^m$ be the map $g(x) = (f(a + \epsilon
    x) - f(a)) \delta^{-1}$.  The multiplication by
    $\epsilon'/\epsilon$ map $\Oo^n \to \Oo^n$ is splendid, so the
    composition
    \begin{equation*}
      g(x\epsilon'/\epsilon) = (f(a + \epsilon'x) - f(a)) \delta^{-1}
    \end{equation*}
    is splendid, showing that $(\epsilon',\delta)$ is suitable.
  \item Similar, using the fact that the multiplication by
    $\delta/\delta'$ map $\Oo^m \to \Oo^m$ is splendid.
  \item If $v(\delta) \le v(\delta_0)$, then $(\epsilon_0,\delta)$ is
    splendid by the previous point.  Suppose $v(\delta) \ge
    v(\delta_0)$.  Let $\rho = \delta/\delta_0 \in \Oo$.  Because the
    pair $(\epsilon_0,\delta_0)$ is suitable, the following function
    is splendid:
    \begin{equation*}
      g(x) = (f(a + \epsilon_0 x) - f(a)) \delta_0^{-1}.
    \end{equation*}
    Note that $g(0) = 0$.  By Lemma~\ref{scaler}, the composition
    \begin{equation*}
      g(\rho x)\rho^{-1} = (f(a + \rho \epsilon_0 x) - f(a)) \delta^{-1}
    \end{equation*}
    is splendid.  Thus $(\rho \epsilon_0, \delta)$ is suitable.
    \qedhere
  \end{enumerate}
\end{proof}
Lemma~\ref{sanity2}(3) shows that in the definition of ``locally
splendid'' (Definition~\ref{ls-def}), we could have said ``for every
$\delta$ there is $\epsilon$'' rather than ``there exist $\delta$ and
$\epsilon$''.
\begin{proposition} \label{ls-local}
  Let $U, U', U_1, \ldots, U_k$ be open sets in $\Mm^n$.
  \begin{enumerate}
  \item If $a \in U' \subseteq U$ and $f : U \to \Mm^m$ is definable,
    then $f$ is locally splendid at $a$ iff $f \restriction U'$ is
    locally splendid at $a$.
  \item If $U = U_1 \cup \cdots \cup U_k$ and $f : U \to \Mm^m$ is
    definable, then $f$ is locally splendid if and only if $f
    \restriction U_i$ is locally splendid for each $i$.
  \end{enumerate}
\end{proposition}
\begin{proof}
  Clearly (1) implies (2).  We prove (1).  First suppose $f$ is
  splendid at $a$.  Take $(\epsilon,\delta)$ which is suitable for $f$
  in the sense of Lemma~\ref{sanity2}.  Take $\epsilon'$ so small that
  $a + \epsilon' \Oo^n \subseteq U'$ and $v(\epsilon') \ge
  v(\epsilon)$.  By Lemma~\ref{sanity2}(1), $(\epsilon',\delta)$ is
  suitable for $f$, which implies it is suitable for $f \restriction
  U'$.

  Conversely, suppose $f \restriction U$ is splendid at $a$.  If
  $(\epsilon,\delta)$ is suitable for $f \restriction U$, then
  $(\epsilon,\delta)$ is suitable for $f$.
\end{proof}
\begin{lemma} \label{future-use}
  Suppose $f : U \to \Mm^m$ is locally splendid at $a \in U$.  If $f,
  U, a$ are $M$-definable for some small $M \preceq \Mm$, then we can
  take the $(\epsilon,\delta)$ witnessing local splendidness to be
  $M$-definable.
\end{lemma}
\begin{proof}
  Take some $(\epsilon_0,\delta_0)$ which is suitable for $f$ at
  $a$.  Let $\mathcal{F}$ be a splendid family containing the splendid
  function
  \begin{equation*}
    x \mapsto (f(a + \epsilon_0 x) - f(a)) \delta_0^{-1}.
  \end{equation*}
  In particular, $\mathcal{F}$ is 0-definable.  Let $D$ be the set of
  pairs $(\epsilon,\delta)$ such that
  \begin{gather*}
    a + \epsilon \Oo^n \subseteq U \\
    f(a + \epsilon \Oo^n) \subseteq f(a) + \delta \Oo^m,
  \end{gather*}
  and the function
  \begin{equation*}
    x \mapsto (f(a + \epsilon x) - f(a)) \delta^{-1}
  \end{equation*}
  is in $\mathcal{F}$.  Then $D$ is $M$-definable and contains
  $(\epsilon_0,\delta_0)$.  By Tarski-Vaught, there is some
  $M$-definable pair $(\epsilon,\delta)$ in $D$.  Then
  $(\epsilon,\delta)$ is suitable, i.e., $(\epsilon,\delta)$ witnesses
  local splendidness of $f$ at $a$.
\end{proof}
\begin{remark}
  Recall that a family $\mathcal{F}$ is \emph{ind-definable} if it is
  a small union of definable families.  For example, the family of
  splendid functions is ind-definable by definition.  The family of
  locally splendid functions is also ind-definable.  Since we will not
  need this fact, we leave the proof as an exercise to the reader.
\end{remark}
Locally splendid functions are closed under similar operations as
splendid functions:
\begin{proposition} \phantomsection \label{ls-props}
  Let $U \subseteq \Mm^n$ be open and $f : U \to \Mm^m$ be definable.
  \begin{enumerate}
  \item If $f$ is a polynomial map, then $f$ is locally splendid.
  \item If $f : U \to V$ is locally splendid and $g : V \to \Mm^\ell$
    is locally splendid, then $g \circ f : U \to \Mm^\ell$ is locally
    splendid.
  \item $f$ is locally splendid iff the component functions
    $f_1,\ldots,f_m : U \to \Mm$ are locally splendid.
  \item The set of splendid functions $U \to \Mm$ is an $\Mm$-algebra,
    i.e., closed under the ring operations and multiplication by
    $\Mm$.
  \item If $f : U \to \Mm^m$ is locally splendid, then $f$ is strictly
    differentiable and the strict derivative $Df : U \to \Mm^{nm}$ is
    locally splendid.
  \end{enumerate}
\end{proposition}
\begin{proof}
  \begin{enumerate}
  \item Given $a \in U$, take $\epsilon$ so small that $a + \epsilon
    \Oo^n \subseteq U$.  The map
    \begin{equation*}
      x \mapsto f(a+\epsilon x) - f(a)
    \end{equation*}
    is polynomial, so if we take $\delta$ very large, then $x \mapsto
    (f(a+ \epsilon x) - f(a))\delta^{-1}$ will be polynomial with
    coefficients in $\Oo$, and therefore splendid.
  \item Take any non-zero $\delta_3$.  By Lemma~\ref{sanity2}(3) there
    is $\delta_2$ such that
    \begin{gather*}
      f(a) + \delta_2 \Oo^m \subseteq V \\
      g(f(a) + \delta_2 \Oo^m) \subseteq g(f(a)) + \delta_3 \Oo^\ell
    \end{gather*}
    and the following map is splendid.
    \begin{equation*}
      g : (f(a) + \delta_2 \Oo^m) \to (g(f(a)) + \delta_3 \Oo^\ell)
    \end{equation*}
    By another application of Lemma~\ref{sanity2}(3) there is
    $\delta_1$ such that
    \begin{gather*}
      a + \delta_1 \Oo^n \subseteq U \\
      f(a + \delta_1 \Oo^n) \subseteq f(a) + \delta_2 \Oo^m
    \end{gather*}
    and the following map is splendid
    \begin{equation*}
      f : (a + \delta_1 \Oo^n) \to (f(a) + \delta_2 \Oo^m).
    \end{equation*}
    Then
    \begin{gather*}
      a + \delta_1 \Oo^n \subseteq U \\
      g(f(a + \delta_1 \Oo^n)) \subseteq g(f(a) + \delta_2 \Oo^m) \subseteq g(f(a)) + \delta_3 \Oo^\ell
    \end{gather*}
    and the composition
    \begin{equation*}
      g \circ f : (a + \delta_1 \Oo^n) \to (g(f(a)) + \delta_3 \Oo^\ell)
    \end{equation*}
    is splendid.  Then $(\delta_1,\delta_3)$ shows that $g \circ f$ is
    locally splendid at $a$.
  \item If $f$ is locally splendid, then the components of $f$ are
    locally splendid by the previous two points (compose $f$ with the
    coordinate projections $\Mm^m \to \Mm$).  Conversely, suppose the
    components $f_1,\ldots,f_m$ are locally splendid.  Fix $a \in U$.
    For each $i$, there is a pair $(\epsilon_i,\delta_i)$ which is
    ``suitable'' for $f_i$ at $a$, in the sense of
    Lemma~\ref{sanity2}.  Take $\epsilon$ and $\delta$ such that
    $v(\epsilon) \ge \max_i v(\epsilon_i)$ and $v(\delta) \le \min_i
    v(\delta_i)$.  Then $(\epsilon,\delta)$ is suitable for each $f_i$
    at $a$ by Lemma~\ref{sanity}(1,2).  In particular,
    \begin{gather*}
      a + \epsilon \Oo^n \subseteq U \\
      f_i(a + \epsilon \Oo^n) \subseteq f_i(a) + \delta \Oo \tag{$\ast$}
    \end{gather*}
    and the map $f_i : (a + \epsilon \Oo^n) \to (f_i(a) + \delta \Oo)$
    is splendid.  Equation ($\ast$) implies $f(a + \epsilon \Oo^n)
    \subseteq f(a) + \delta \Oo^m$.  The map $f : (a + \epsilon \Oo^n)
    \to (f(a) + \delta \Oo^m)$ is splendid because each component is
    splendid.
  \item This follows from the previous three points.
  \item Strict differentiability is clear because, up to a change of
    coordinate, $f$ looks locally like a splendid function, and
    splendid functions are strictly differentiable.  Locally, the
    strict derivative is a scaled version of the strict derivative of
    a splendid function, so the strict derivative is locally
    splendid. \qedhere
  \end{enumerate}
\end{proof}

\subsection{Generic local splendidness}
Next, we work towards proving that definable functions are generically
locally splendid.  To do this, we will essentially show that the
theory of $p$-adically closed fields has locally splendid definable
Skolem functions.

Say that a function on $K$ is \emph{strongly analytic} if it is given
by a single convergent power series, and \emph{analytic} if it is
locally strongly analytic.  (Some authors say ``locally analytic''
rather than ``analytic''.)
\begin{lemma} \label{analytic-ls}
  Let $f : U \to \Mm^m$ be a 0-definable function such that $U
  \subseteq \Mm^n$ is open, $U(K)$ is compact, and $f(K) : U(K) \to
  K^m$ is analytic.  Then $f$ is locally splendid.
\end{lemma}
\begin{proof}
  Because $K$ is locally compact and $U(K)$ is compact, we can cover
  $U$ with $K$-definable balls $U = B_1 \cup \cdots \cup B_\ell$ such
  that $f(K)$ is strongly analytic on each $B_i(K)$.  By the
  assumption on $\Ll$ at the start of Section~\ref{thpr}, the balls
  $B_i$ are 0-definable.  Each of the restrictions $f(K) \restriction
  B_i(K)$ is strongly analytic, given by a convergent power series.
  Replacing $f$ with $\delta \cdot f$ for a small enough non-zero
  $\delta \in \dcl(\varnothing)$, we can arrange that all the
  coefficients in these power series are in $\Oo_K$.  Then $f(K)
  \restriction B_i(K)$ is (essentially) a 0-definable pre-splendid
  function for each $i$, and so $f \restriction B_i$ is splendid
  (Proposition~\ref{lazy}).  By Proposition~\ref{ls-local}(2), $f$ is
  locally splendid.
\end{proof}
\begin{lemma} \label{toto}
  The field operations are locally splendid on their domains.
\end{lemma}
\begin{proof}
  Addition, multiplication, and subtraction are polynomials, so they
  are locally splendid by Proposition~\ref{ls-props}(1).  It remains
  to show that the map $f(x) = 1/x$ is locally splendid on
  $\Mm^\times$.  Lemma~\ref{analytic-ls} shows that $f$ is locally
  splendid on the subset $\Oo^\times$.  Then $f : c\Oo^\times \to
  c^{-1}\Oo^\times$ is also locally splendid for any $c \in
  \Mm^\times$, because it's the composition
  \begin{equation*}
    c\Oo^\times \stackrel{x \mapsto x/c}{\longrightarrow} \Oo^\times \stackrel{f}{\longrightarrow} \Oo^\times \stackrel{x \mapsto x/c}{\longrightarrow} c^{-1}\Oo^\times
  \end{equation*}
  and the two maps on the outside are polynomial maps.  Then we have
  covered $\dom(f)$ with sets $c\Oo^\times$ on which $f$ is locally
  splendid, so $f$ is locally splendid.
\end{proof}
\begin{lemma} \phantomsection \label{argh}
  \begin{enumerate}
  \item Let $\mm$ denote the maximal ideal.  Let
    \begin{equation*}
      h_n : \mm \times \Oo^\times \times \Oo^{n-2} \to \mm
    \end{equation*}
    be the definable function mapping $(a_0,a_1,\ldots,a_{n-1})$ to
    the unique root in $\mm$ of
    \begin{equation*}
      x^n + a_{n-1}x^{n-1} + \cdots + a_0 = 0,
    \end{equation*}
    the root guaranteed to exist by Hensel's lemma.  Then $h_n$ is
    locally splendid.
  \item Let $P_n$ denote the set of non-zero $n$th powers.  Then there
    is a locally splendid 0-definable map
    \begin{gather*}
      P_n \to \Mm \\
      x \mapsto \sqrt[n]{x}
    \end{gather*}
    assigning an $n$th root to each $x \in P_n$.
  \end{enumerate}
\end{lemma}
In (2), we are not claiming that $\sqrt[n]{x}$th is a $\frac{1}{n}$th
power map in the sense of Definition~\ref{qp-def}, i.e.,
$\sqrt[n]{xy}$ need not equal $\sqrt[n]{x}\sqrt[n]{y}$.
\begin{proof}
  \begin{enumerate}
  \item The function $h_n(K) : \mm_K \times \Oo_K^\times \times
    \Oo_K^{n-2} \to \mm_K$ is analytic, by the implicit function
    theorem for analytic functions.  (See
    \cite[Proposition~5.9]{schneider}, at least for the \emph{inverse}
    function theorem for analytic functions, which easily implies the
    implicit function theorem.)  Then $h_n$ is locally splendid by Lemma~\ref{analytic-ls}.
  \item By Lemma~\ref{qth-power}, there is some $m$ such that $P_m$
    has an $n$th power map $x^{1/n}$ in the sense of
    Definition~\ref{qp-def}.  Increasing $m$, we may assume $n \mid
    m$.  We first show that this map $x^{1/n}$ on $P_m$ is locally
    splendid.

    The map $x^{1/n}$ is analytic on $P_m(K)$ by the inverse function
    theorem for analytic functions.\footnote{Alternatively, $x^{1/n}$
    is analytic at almost all points by the generic analyticity of
    definable functions in $K$ \cite[Theorem~1.1]{vdDS}.  Since
    $x^{1/n}$ is a homomorphism, as soon as it is locally analytic at
    one point, it is analytic everywhere.}  However, $P_m(K)$ is not
    compact, so we cannot immediately apply Lemma~\ref{analytic-ls}.
    Nevertheless, the subset $P_m(K) \cap \Oo_K^\times$ is compact, so
    we at least see that $x^{1/n}$ is locally splendid on $P_m \cap
    \Oo^\times$ by Lemma~\ref{analytic-ls}.  Using the same method as
    in Lemma~\ref{toto}, this implies that $x^{1/n}$ is locally
    splendid on $P_m \cap c\Oo^\times$ for every $c \in P_m \cap
    \Oo^\times$, and so $x^{1/n}$ is locally splendid on $P_m$.

    Finally, take $a_1, \ldots, a_k$ coset representatives of $P_m$ in
    $P_n$, and take $b_i$ to be any $n$th root of $a_i$.  By definable
    Skolem functions, we can take $a_i, b_i \in \dcl(\varnothing)$.
    Finally, define
    \begin{gather*}
      \sqrt[n]{-} : P_n \to \Mm^\times \\
      \sqrt[n]{x} = b_i (x/a_i)^{1/n} \text{ if } x \in a_i P_m.
    \end{gather*}
    The map $\sqrt[n]{x}$ is defined by gluing together locally
    splendid functions on the cosets $a_1P_m, a_2P_m, \ldots, a_kP_m$,
    so $\sqrt[n]{x}$ is locally splendid by
    Proposition~\ref{ls-local}(2).  \qedhere
  \end{enumerate}
\end{proof}
\begin{corollary} \label{dcl}
  Let $M \prec \Mm$ be a small model.  Let $\ba$ be an $n$-tuple in
  $\Mm$.  Let $b$ be an element in $\dcl(\ba M)$.  Then $b = f(\ba)$
  for some $M$-definable locally splendid function $f : U \to \Mm$
  with $\ba \in U \subseteq \Mm^n$.
\end{corollary}
\begin{proof}
  Let $A$ be the set of elements of the form $f(\ba)$, where $f$ is
  $M$-definable and locally splendid.  It suffices to show that
  $\dcl(\ba M) \subseteq A$.  If $g$ is $M$-definable and locally
  splendid and $\bb \in A^n$, then $g(\bb) \in A$ because a
  composition of locally splendid functions is locally splendid
  (Proposition~\ref{ls-props}).  By Lemma~\ref{toto}, $A$ is a
  subfield of $\Mm$.  As constant functions and coordinate projections
  are locally splendid, $A$ contains $M$ and the coordinates of the tuple $\ba$.  By
  Lemma~\ref{argh}(1), $A$ is a Henselian valued field.  Lastly,
  Lemma~\ref{argh}(2) shows that if $a \in A$ and $a \in P_n =
  P_n(\Mm)$, then $\sqrt[n]{a}$ exists in $A$.
  \begin{claim}
    If $\gamma \in \Gamma_A$ and $\gamma$ is a multiple of $n$ in
    $\Gamma_\Mm$, then $\gamma/n \in \Gamma_A$.
  \end{claim}
  \begin{claimproof}
    Take an element $b \in A$ with $v(b) = \gamma$.  Because $P_n$ has
    finite index in the multiplicative group, it has only finitely
    many cosets, and all of them are $M$-definable.  In particular, $b
    P_n$ is $M$-definable, so it contains some $M$-definable element
    $c$.  Then $b/c \in P_n$.  It follows that $v(b/c)$ is a multiple
    of $n$ (in $\Gamma_\Mm$).  As $v(b) = \gamma$ is a multiple of
    $n$, we also see that $v(c)$ is a multiple of $n$ (in $\Gamma_M$
    or $\Gamma_\Mm$).  Then
    \begin{equation*}
      \frac{\gamma}{n} = \frac{v(b)}{n} = \frac{v(b/c) + v(c)}{n} = \frac{v(b/c)}{n} + \frac{v(c)}{n} = v\left(\sqrt[n]{b/c}\right) + \frac{v(c)}{n}.
    \end{equation*}
    Because $A$ is a field containing $M$, we have $b/c \in A$.
    Because $A$ is closed under $n$th roots (when they exist), we have
    $\sqrt[n]{b/c} \in A$.  Then $v\left(\sqrt[n]{b/c}\right) \in
    \Gamma_A$.  Finally, $v(c)/n \in \Gamma_M \subseteq \Gamma_A$.  It
    follows that $\gamma/n \in \Gamma_A$, proving the Claim.
  \end{claimproof}
  Then for any $n$, the map $\Gamma_A/n\Gamma_A \to
  \Gamma_{\Mm}/n\Gamma_{\Mm}$ is injective.  Since the
  composition \[\Zz/n\Zz \cong \Gamma_M/n\Gamma_M \to
  \Gamma_A/n\Gamma_A \to \Gamma_{\Mm}/n\Gamma_{\Mm} \cong \Zz/n\Zz\]
  is a bijection, it follows that $\Gamma_A/n\Gamma_A \cong \Zz/n\Zz$.
  The group $\Gamma_A$ is also discretely ordered, as it contains the
  minimal positive element of $\Gamma_M$.  Then $\Gamma_A$ is a model
  of Presburger arithmetic (a $\Zz$-group).

  To summarize, $A$ is a henselian valued field sitting between $M$
  and $\Mm$, and its value group is a $\Zz$-group.  By the
  axiomatization of $\Th(K)$ plus its model completeness, we see that
  $M \preceq A \preceq \Mm$.  Then $\dcl(M\bar{a}) \subseteq A$.
\end{proof}
Using this, we see that definable functions are generically locally
splendid:
\begin{proposition} \label{gen-lsa}
  Let $U$ be a non-empty definable open set in $\Mm^n$ and $f : U \to
  \Mm^m$ be definable.  Then there is a smaller definable open set
  $U_0 \subseteq U$ such that $f \restriction U_0$ is locally
  splendid, and $\dim(U \setminus U_0) < \dim(U)$.  Moreover, if $M$
  is a small model defining $f$ and $U$, then we can take $U_0$ to be
  $M$-definable.
\end{proposition}
\begin{proof}
  Fix a small $M \preceq \Mm$ defining $U$ and $f$, if none was given.
  \begin{claim}
    If $\ba \in U$, then one of the following holds:
    \begin{itemize}
    \item $\ba \in D$ for some $M$-definable set $D \subseteq \Mm^n$
      with $\dim(D) < n$.
    \item $\ba \in D$ for some $M$-definable open set $D \subseteq
      \Mm^n$ such that $f \restriction D$ is locally splendid.
    \end{itemize}
  \end{claim}
  \begin{claimproof}
    This follows from Corollary~\ref{dcl}.  If $\dim(\ba/M) < n$, then
    the first case holds.  Suppose that instead, $\dim(\ba/M) = n$.
    Then Corollary~\ref{dcl} shows that $f(\ba) = g(\ba)$ for some
    locally splendid $M$-definable function $g$.  Let $D_0 =
    \{\bx \in \dom(f) \cap \dom(g) : f(\bx) = g(\bx)\}$.  Then $D_0$
    is $M$-definable and $\ba \in D_0$, so $\dim(D_0) = n$.  Let $D$
    be the interior of $D_0$.  Then $\dim (D_0 \setminus D) < n$, so
    $\ba \notin D_0 \setminus D$, and instead $\ba \in D$.  The
    restriction $f \restriction D$ equals the locally splendid
    function $g \restriction D$.
  \end{claimproof}
  By the Claim and saturation, we can cover $U$ with finitely many
  $M$-definable sets
  \begin{equation*}
    D_1 \cup D_2 \cup \cdots \cup D_k \cup U_1 \cup \cdots \cup U_\ell
  \end{equation*}
  such that $\dim(D_i) < n$, $U_i$ is open, and $f \restriction U_i$
  is locally splendid.  Let $U_0 = U_1 \cup \cdots \cup U_\ell$.
  Then
  \begin{equation*}
    \dim (U \setminus U_0) \le \dim \bigcup_i D_i < n = \dim(U),
  \end{equation*}
  and $f \restriction U$ is locally splendid by
  Proposition~\ref{ls-local}(2).
\end{proof}
%% TODO: simplify the arguments.  This section has so many steps that are
%% ``obvious'' to experts; I can omit most of them.

\subsection{Back to groups}
If $(G,\star)$ is a definable group and $a \in G$, recall the operation
\begin{equation*}
  x \star_a y = x \star a^{-1} \star y
\end{equation*}
from Section~\ref{again} making $G$ into a definable group with
identity element $a$, definably isomorphic to the original group.
\begin{lemma} \label{la-2}
  Let $M$ be a small model.  Let $G \subseteq \Mm^n$ be an
  $n$-dimensional $M$-definable set and $\star$ be an $M$-definable
  group operation on $G$.  Then there is an $M$-definable set $U$ in
  the interior of $G$ with $\dim (G \setminus U) < n$, such that for
  any $a \in U$, the group operation $\star_a$ is locally splendid at
  $(a,a)$.
\end{lemma}
\begin{proof}
  Like the ``direct'' proof of Lemma~\ref{local-C1}, using
  Proposition~\ref{gen-lsa} to get an $M$-definable set $\Delta \subseteq
  G \times G$ on which the group operation is locally splendid.
\end{proof}

Recall the condition $\mathfrak{C}_\omega$ from Definition~\ref{ABC}.
\begin{lemma} \label{mess}
  Let $G \subseteq \Mm^n$ be an $n$-dimensional definable set and let
  $\star$ be a definable group operation on $G$ with identity element
  $\bar{0}$.  Let $\epsilon, \delta$ be non-zero elements of $\Mm$, with
  $\epsilon \Oo^n$ and $\delta \Oo^n$ contained in $G$.  Suppose that
  the group operation $\star$ maps $\epsilon \Oo^n \times \epsilon
  \Oo^n$ into $\delta \Oo^n$, and the map
  \begin{gather*}
    f : \Oo^n \times \Oo^n \to \Oo^n \\
    f(\bx,\by) = \delta^{-1} \cdot (\epsilon \bx \star \epsilon \by)
  \end{gather*}
  is splendid.  Let
  \begin{equation*}
    \rho = p \epsilon^2/\delta.
  \end{equation*}
  Then $(\rho \Oo^n, \star)$ is a subgroup of $G$ satisfying
  $\mathfrak{C}_\omega$.
\end{lemma}
\begin{proof}
  The set of counterexamples is ind-definable, because the family of
  splendid functions is ind-definable (by definition) and the
  $\mathfrak{C}_\omega$ condition is type-definable
  (Remark~\ref{def-el}).  The set of counterexamples is even
  ind-definable over the empty set, by automorphism invariance.  If it
  is non-empty, it must contain a $K$-definable point.  In other
  words, if there is a counterexample, then there is a $K$-definable
  counterexample.

  Therefore, we may assume that $G, \star, \epsilon, \delta$ are
  $K$-definable.  Then $f$ is a $K$-definable splendid map, so $f(K) :
  \Oo_K^n \times \Oo_K^n \to \Oo_K^n$ is pre-splendid, given by a
  power series
  \begin{equation*}
    f(\bx,\by) = \sum_{I,J} \bc_{I,J} \bx^I \by^J
  \end{equation*}
  with $\bc_{I,J} \in \Oo_K^m$.  Then the group law itself is given by
  \begin{align*}
    \bx \star \by &= \delta \cdot f(\epsilon^{-1} \bx, \epsilon^{-1} \by) \\
    &= \sum_{I,J} \bc_{I,J} \delta \epsilon^{-|I|-|J|} \bx^I \by^J,
  \end{align*}
  at least for $\bx, \by \in \epsilon \Oo_K^n$.  By the identity law for
  $\star$, the first few terms of the power series for $\star$ must be
  $\bx + \by + \ldots$, and so the power series for $f$ must have the
  form
  \begin{equation*}
    f(\bx,\by) = \frac{\epsilon}{\delta} \bx + \frac{\epsilon}{\delta}
    \by + \cdots
  \end{equation*}
  In particular, $\epsilon/\delta \in \Oo$.  It follows that $\rho = p
  (\epsilon/\delta) \epsilon$ has higher valuation than $\epsilon$, so
  \begin{equation*}
    \rho \Oo^n \subseteq \epsilon \Oo^n \subseteq G.
  \end{equation*}
  If $\bx, \by \in \Oo_K^n$, then
  \begin{align*}
    \rho^{-1}(\rho \bx \star \rho \by) &= \sum_{I,J} \bc_{I,J} \rho^{|I|+|J|-1} \delta \epsilon^{-|I|-|J|} \bx^I \by^J \\
    &= \bx + \by + \frac{\delta}{\epsilon} \sum_{\substack{I,J \\ |I| \ge 1 \\ |J| \ge 1}}  \bc_{I,J} \left(\frac{\rho}{\epsilon}\right)^{|I|+|J|-1}  \bx^I \by^J.
  \end{align*}
  By Example~\ref{c-source}, it remains to show that
  \begin{equation*}
    \frac{\delta}{\epsilon} \left(\frac{\rho}{\epsilon}\right)^{|I|+|J|-1} \bc_{I,J} \stackrel{?}{\in} p^{|I|+|J|-1} \Oo^n.
  \end{equation*}
  But this is clear, because
  \begin{equation*}
    \frac{\delta}{\epsilon} \left(\frac{\rho}{\epsilon}\right)^{|I|+|J|-1} = \frac{\delta}{\epsilon} \left(p \frac{\epsilon}{\delta} \right)^{|I|+|J|-1} = p^{|I|+|J|-1} (\epsilon/\delta)^{|I|+|J|-2} \in p^{|I|+|J|-1} \Oo^n,
  \end{equation*}
  using the fact that $|I|+|J| \ge 2$ and $\epsilon/\delta \in \Oo$.
\end{proof}

\begin{theorem} \label{last-main}
  Let $M \preceq \Mm$ be a model (of $T = \Th(K)$) and let $G$ be an
  $M$-definable group.  There is an $M$-definable open subgroup $H
  \subseteq G$ such that
  \begin{itemize}
  \item $H$ is isomorphic to a definable group $(\Oo^n,\star)$
    satisfying $\mathfrak{C}_\omega$.
  \item $H$ is fsg and compactly dominated by $H/H^{00}$, and
    $H/H^{00}$ is isomorphic to an $n$-dimensional $p$-adic Lie group.
  \end{itemize}
\end{theorem}
\begin{proof}
  By Lemma~\ref{stick} we may assume $G \subseteq \Mm^n$.  By
  Lemma~\ref{la-2}, there is an $M$-definable non-empty set $U$ such
  that $\star_a$ is locally splendid around $(a,a)$ for any $a \in U$.
  Fix an $a \in U(M)$.  Then $(G,\star)$ is definably isomorphic to
  $(G,\star_a)$.  Replacing $(G,\star)$ with $(G,\star_a)$, we may
  assume that $\star$ is locally splendid at the identity element
  $1_G$.  Translating $G$, we may assume the identity element is
  $\bar{0} \in \Mm^n$.

  By definition of ``locally splendid,'' there are $\epsilon$ and
  $\delta$ such that $\star$ maps $\epsilon \Oo^n \times \epsilon
  \Oo^n$ into $\delta \Oo^n$, and the map
  \begin{gather*}
    f : \Oo^n \times \Oo^n \to \Oo^n \\
    f(\bx,\by) = \delta^{-1} \cdot (\epsilon \bx \star \epsilon \by)
  \end{gather*}
  is splendid.  We can choose $\epsilon, \delta \in
  M$ by Lemma~\ref{future-use}.  Then
  Lemma~\ref{mess} gives an ($M$-definable) subgroup $\rho \Oo^n
  \subseteq G$ satisfying $\mathfrak{C}_\omega$.  By
  Proposition~\ref{c-good}, the subgroup $\rho \Oo^n$ has the desired
  properties.
\end{proof}

\begin{acknowledgment}
  The author was supported by the National Natural Science Foundation
  of China (Grant No.\@ 12101131) and the Ministry of Education of
  China (Grant No.\@ 22JJD110002).  The author would like to thank the anonymous referee, who read the paper extremely carefully and offered many helpful comments, especially concerning the history of $P$-minimality and $p$-adically closed fields.
\end{acknowledgment}

\bibliographystyle{plain} \bibliography{mybib}{}

\end{document}